\documentclass[referee,pdflatex,sn-mathphys-ay]{sn-jnl}
\usepackage{graphicx} % Required for inserting images
\usepackage[utf8]{inputenc}
\usepackage{amsmath,amsbsy,amscd,amssymb,amsthm,graphicx,epsfig,color}
\usepackage{mathtools}
\usepackage{enumitem}
\usepackage[english]{babel}
\usepackage{multicol}
\usepackage{multirow}
\usepackage{subcaption}
\usepackage{bigints}
\usepackage{tabularx}
\usepackage{xspace}
\usepackage{natbib}
\usepackage{booktabs}
\usepackage[title]{appendix}%

\usepackage{listings}%

\theoremstyle{thmstyleone}%
\newtheorem{theorem}{Theorem}

\newcommand{\x}{\mathbf x}
\newcommand{\q}{\mathbf q}
\newcommand{\f}{\mathbf f}
\newcommand{\fq}{\f^{\mathrm q}}
\newcommand{\diff}{\mathrm d}

\newcommand{\tmax}{\mathrm t_{\max}}
\newcommand{\tm}{\mathrm t_{\mathrm m}}
\newcommand{\sign}{\mathrm{sign}}

\newcommand{\R}{\mathbb{R}}

\newcommand{\kQSS}[1]{k^{(#1)}_{\mathrm{QS}(t_0,t_f)}}

\newcommand{\dQ}{\Delta \mathbf{Q}}
\newcommand{\dQabs}{\Delta Q_{\mathrm{abs}}}
\newcommand{\dQrel}{\Delta Q_{\mathrm{rel}}}

\newcommand{\lacts}[1] {A_{x_i}^{(#1)}}

\newcommand{\ilact}[1] {a_{x_i}^{(#1)}}

\usepackage{xparse}
\newcommand{\CQSS}{CheQSS\xspace}

\date{}

\title[On General LIQSS Methods]{On General Linearly Implicit Quantized State System Methods}
\author*[1,2]{\fnm{Mariana} \sur{Bergonzi}}\email{bergonzi@cifasis-conicet.gov.ar}
\author[1]{\fnm{Joaquín} \sur{Fernandez}}\email{fernandez@cifasis-conicet.gov.ar}
\author[1,2]{\fnm{Ernesto} \sur{Kofman}}\email{kofman@cifasis-conicet.gov.ar}

\affil*[1]{\orgdiv{French-Argentine International Center for Information and System Sciences (CIFASIS)}, \orgname{CONICET}, \orgaddress{\country{Argentina}}}
\affil[2]{\orgdiv{FCEIA}, \orgname{Universidad Nacional de Rosario}, \orgaddress{\country{Argentina}}}

\begin{document}
\abstract{This work proposes a methodology to develop new numerical integration algorithms for ordinary differential equations based on state quantization, generalizing the notions of Linearly Implicit Quantized State Systems (LIQSS) methods. Using this idea, two novel sub-families of algorithms are designed that improve the performance of current LIQSS methods while preserving their properties regarding stability, global error bound and efficient event handling capabilities. 
The features of the new algorithms are studied in two application examples where the advantages over classic numerical integration algorithms is also analyzed.}

\keywords{Quantized State Systems, Numerical Integration Algorithms, Computational Efficiency}



\maketitle

\section{Introduction}
The class of numerical ODE solvers known as Quantized State System (QSS) methods replaces the conventional time stepping approach with the quantization of system states for integration \citep{kofman2001quantized,cellier2006continuous}. This approach yields asynchronous, discrete-event simulation frameworks that frequently provide substantial benefits when compared against traditional numerical differential equation solvers.

The fundamental principle of QSS techniques involves recalculating a state variable solely upon its exhibiting a predefined deviation. Consequently, for systems displaying heterogeneous activity \citep{bergonzi2025activity}—where certain states undergo major fluctuation as others remain largely constant—QSS algorithms efficiently leverage this attribute by executing operations exclusively for the changing states. Furthermore, the event-driven operation inherent to these methods significantly eases the challenge of managing discontinuities.

Given these properties, QSS methodologies are generally well-applied across various fields, such as the simulation of power electronic converters \citep{migoni2015quantization}, tackling advection-diffusion-reaction problems \citep{bergero2016time}, and simulating spiking neural architectures \citep{bergonzi2023quantization}, among others. In studies concerning these domains, QSS-based simulations can often achieve computation speeds exceeding tenfold faster than their most performant classic ODE counterparts.

The family of QSS methods is composed of the first-order QSS1 \citep{kofman2001quantized}, the second-order QSS2 \citep{kofman2002second}, and the third-order QSS3 \citep{kofman2006third} algorithms. In addition, there are Linearly Implicit QSS methods of orders 1 to 3 \citep{migoni2013linearly} that are able to simulate efficiently some types of stiff systems.

In this work, we develop a new methodology to design LIQSS methods that is based on proposing the polynomials that compute the difference between the states and the quantized states of the system. Using this methodology, we retrieve and improve the original formulation of LIQSS methods and also develop a new family of methods of orders 1 to 3 that use Chebyshev polynomials and are called \CQSS. Under ideal assumptions, \CQSS methods maximize the step sizes for a given tolerance prescription.

The novel methods are then evaluated in two simulation examples: a spatially-discretized one-dimensional advection-diffusion-reaction equation and a spiking neural network. We show that in both examples the novel methods consistently outperform previously existing QSS and LIQSS methods and also outperform classic discrete-time ODE solvers like DOPRI and CVODE.

The manuscript is organized as follows: Section~\ref{sec:background} provides an introduction to QSS methods and their properties. Then, Section~\ref{sec:main} contains the main results, including the LIQSS design methodology, the novel algorithms and their properties. After that, Section~\ref{sec:results} presents the examples and results, analyzing and comparing the performance of the methods developed. Finally, Section~\ref{sec:conclusions} concludes the work, exploring also some future lines of research.

\section{Background} \label{sec:background}

\subsection{QSS Methods}
Consider a system of ODEs of the form
\begin{equation}\label{eq:ode}
    \dot{\x}_a(t)=\f(\x_a(t),t)
\end{equation}
where $\x_a(t) \in \R^N$ is the state vector,  the first order Quantized State System (QSS1) method \citep{kofman2001quantized} solves an approximate ODE called Quantized State System:
\begin{equation}\label{eq:qss}
    \dot \x(t)=\fq(\q(t),t)
\end{equation}
where $\x=[x_1,x_2,\ldots,x_N]^T$ is the vector of states, $\q=[q_1,q_2,\ldots,q_N]^T$ is the vector of \emph{quantized states} and $\fq(\q,t)$ is a piecewise constant approximation of $\f(\x,t)$ with respect to $t$.

In Eq.\eqref{eq:qss}, each state $x_i(t)$ is approximated by a quantized state $q_i(t)$ that follows a piecewise constant trajectory evolving according to the following rule: the quantized value is updated only when the deviation $|x_i(t)-q_i(t)|$ exceeds a \emph{quantum} $\Delta Q_i$. 
This results in asynchronous, event-driven simulations where each quantized state is updated at its own event times.  

Formally, each state $x_i(t)$ and its corresponding quantized state $q_i(t)$ are related by a \emph{hysteretic quantization function} as follows:  
\begin{equation*}
 q_i(t) =
 \begin{cases}
  q_i(t_k) \quad &\text{if} \;|x_i(t)-q_i(t_k)| < \Delta Q_i	\\
  x_i(t) \quad &\text{otherwise}
 \end{cases}
\end{equation*}
for $t_k< t\leq t_{k+1}$, where $t_{k+1}$ is the first time after $t_k$ such that $|x_i(t)-q_i(t_k)| = \Delta Q_i$.  
Figure~\ref{fig:qss}(a) illustrates the trajectories of a state and the corresponding quantized state.

\begin{figure}[ht]
    \centering
    \includegraphics[width=\linewidth]{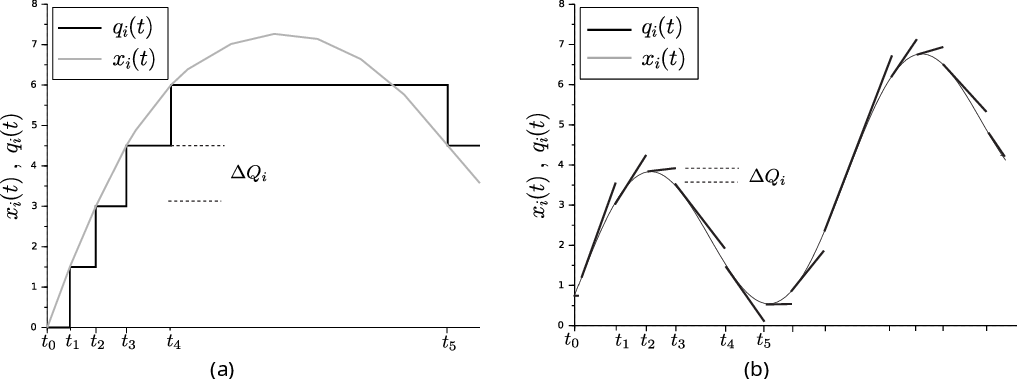}
    \caption{Typical state and quantized state trajectories: (a) QSS1, (b) QSS2}
    \label{fig:qss}
\end{figure}

Since $\q(t)$ follows piecewise constant trajectories and $\fq(\q,t)$ is also piecewise constant on $t$, then the state derivative $\dot \x(t)$ is piecewise constant and the state $\x(t)$ follows piecewise linear trajectories. This allows obtaining the analytic solution of Eq.\eqref{eq:qss} and implementing a simple algorithm for QSS simulation. 

\subsubsection{Higher order QSS methods}
It can be easily noticed that the number of changes in $q_i(t)$ is inversely proportional to $\Delta Q_i$ which in turn implies that the number of steps in a QSS1 simulation increases linearly with the accuracy.

To improve accuracy without increasing linearly the number of steps, higher-order extensions were introduced: in QSS2, quantized states evolve along piecewise linear segments (yielding parabolic state trajectories) as shown in Figure~\ref{fig:qss}(b), while in QSS3 they follow parabolic segments (producing cubic trajectories) \citep{kofman2002second,kofman2004discrete}. In these schemes, the number of steps grows roughly with the square root (QSS2) or cubic root (QSS3) of the desired precision—making them much more efficient than QSS1.  

\subsubsection{Practical Advantages of QSS Methods}
Each QSS step is local to the quantized state $q_i$ that changes its value, and calculations and made only on that state and in the state derivatives that directly depend on it. That way, the computational costs of each QSS step in a large sparse systems is considerably cheaper than that of a classic solver. This is particularly advantageous in system that exhibit heterogeneous activity \citep{bergonzi2025activity} since the calculations are concentrated only in the states that have significant changes. 

QSS methods also handle discontinuities efficiently since the occurrence of a discontinuity does not require to reinitialize the simulation like in classic ODE solvers. Here, each event requires only to recompute the state derivatives that are directly affected by the corresponding discontinuity. Moreover, since trajectories are piecewise polynomials, detecting threshold crossings is straightforward.

\subsubsection{Theoretical Properties of QSS Methods}
The QSS approximation of Eq.\eqref{eq:qss} can be rewritten as
\begin{equation}\label{eq:disturbance}
    \dot \x(t)=\fq(\x(t)+\Delta \x(t),t)
\end{equation}
with $\Delta \x(t)\triangleq \q(t)-\x(t)$. Thus, the use of the quantized state instead of the state in the approximation is equivalent to add a disturbance $\Delta \x(t)$. This disturbance is bounded component-wise by the quantum, since $|\Delta x_i(t)|\leq \Delta Q_i$ because the quantized states $q_i(t)$ and the states $x_i(t)$ never differ from each other more than the quantum $\Delta Q_i$.

Based on this observation, the following properties were established \citep{cellier2006continuous}:
\begin{itemize}
    \item Convergence: Assuming Lipschitz conditions on $\f$, the solutions of the QSS approximation of Eq.\eqref{eq:qss} go to the solutions of the original ODE of Eq.\eqref{eq:ode} when the quantum $\dQ\to 0$ and $\|\fq(\x,t)- \f(\x,t)\| \to 0$.
    \item Global Error Bound: In Linear Time Invariant (LTI) systems the global error is bounded by a quantity that depends linearly on the quantum $\dQ$.
\end{itemize}

In addition, there are recent results that establish convergence properties of QSS methods applied to stochastic differential equations \citep{pizzi2025quantized}.

\subsection{LIQSS Methods}
In presence of stiff systems, QSS methods exhibit trajectories with spurious oscillations similar to those of explicit classic variable step size solver. In consequence, the number of steps they perform is usually unacceptable. 

In order to address stiff dynamics, the QSS family has been further extended  through the development of the Linearly Implicit QSS (LIQSS) algorithms \citep{migoni2013linearly}. 
The core idea behind LIQSS is inspired by implicit integration techniques, which evaluate the state derivatives at future time instants. Instead of solving implicit equations or performing matrix inversions, LIQSS exploits the fact that in QSS methods the future value of each quantized state is known ($q_i(t) \pm \Delta Q_i$). When the future derivative is predicted to be positive, the quantized state is set to $q_i(t)=x_i(t)+ \Delta Q_i$, if negative, it is set to  $q_i(t)=x_i(t)- \Delta Q_i$. Figure~\ref{fig:liqss1_1}(a) shows typical LIQSS1 trajectories. 

\begin{figure}[ht]
    \centering
    \includegraphics[width=\linewidth]{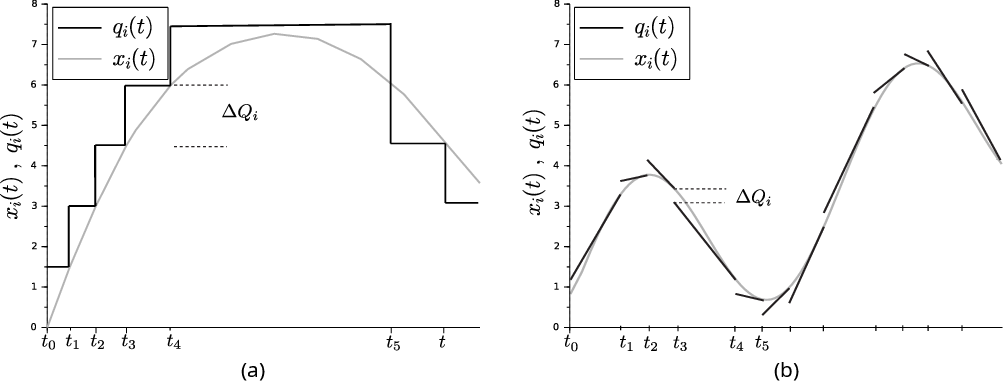}
   \caption{Typical state and quantized state trajectories: (a) LIQSS1, (b) LIQSS2}
   \label{fig:liqss1_1}
 
\end{figure}

It may happen that the selection of a positive increment $q_i(t)=x_i(t)+ \Delta Q_i$ yields a negative state derivative $\dot x_i(t)<0$ while the selection of a negative increment $q_i(t)=x_i(t) - \Delta Q_i$ yields a positive state derivative $\dot x_i(t)>0$.  In that situation, a value $q_i^*\in(x_i(t) - \Delta Q_i,x_i(t) + \Delta Q_i)$ must exist for which $\dot x_i(t)=0$, so the quantized state is selected as $q_i(t)=q_i^*$. In this \emph{equilibrium} situation the algorithm, in principle, does not need to perform more steps. 

The prediction of the sign of the state derivative and the solution for $q_i^*$ are obtained using a linear approximation of the components of Eq.\eqref{eq:qss} given by 
\begin{equation}\label{eq:dxi}
 \dot x_i(t)=\f^{\mathrm q}_i(\q(t),t)=a_i(t^i_j) q_i(t)+ u_i(t)
\end{equation}
where $t^i_j$ is the time of the $j$--th quantized state change in $q_i(t)$, i.e., $t^i_j\leq t < t^i_{j+1}$, and 

\begin{equation}\label{eq:linearap}
 a_i(t^i_j)=\frac{\partial f_i}{\partial x_i}\bigg\rvert_{\q(t^i_j)},\quad 
 u_i(t)=\f^{\mathrm q}_i(\q(t),t)-a_i(t^i_j) q_i(t)
\end{equation}

\subsubsection{Higher Order LIQSS Methods}
Higher-order versions (LIQSS2 and LIQSS3) extend these ideas to piecewise linear and quadratic quantized state trajectories, combining the advantages of higher-order QSS methods with improved performance for stiff problems. 

It is worth mentioning that the definition of LIQSS2 given in \citep{migoni2013linearly} computes the quantized state such that each linear segment  starts at $q_i(t_j)=x_i(t_j)\pm \Delta Q_i$ verifying also that $\dot q_i(t_j)=\dot x_i(t_j)$ (so $q_i$ had a future value of $x_i$ but not a future slope). This was later improved in \citep{di2019improving} where  the quantized state was computed such that  $\dot q_i(t_j+h)=\dot x_i(t_j+h)$ where $t_j+h$ is the time of the next change in $q_i$. Figure~\ref{fig:liqss1_1}(b) shows typical LIQSS2 trajectories.  A similar idea is mentioned regarding LIQSS3, although it was not implemented in practice. 

\subsubsection{Theoretical properties of LIQSS methods}
In the original formulation of LIQSS methods, the quantized state $q_i(t)$ was allowed to differ up to $2\cdot \Delta Q_i$ from the state $x_i(t)$ in case the state trajectory changes its direction and starts moving away $q_i(t)$. For this reason, the disturbance introduced by the difference $q_i(t)-x_i(t)$ can be twice that of QSS methods and consequently the error bounds are also twice those of QSS algorithms.

\subsubsection{Limitations of LIQSS for Stiff Systems Simulation}
The fact that the linear approximation is only performed using the main diagonal of the Jacobian matrix implies that the LIQSS algorithms are not able to efficiently handle general stiff systems. LIQSS are limited to efficiently integrate systems where  the stiffness is due to the presence of large terms in the main diagonal of the Jacobian matrix. 

This feature was improved in \citep{di2019improving} by detecting spurious oscillations appearing between pairs of variables and updating simultaneously both quantized states with a linearly implicit scheme.  These ideas were further improved in \citep{elbellili2025improving}.   

\subsection{Relative Error Control}
The use of a constant quantum is equivalent to perform absolute error control in classic discrete time algorithms. However, in several applications the relative error needs to be controlled. This can be achieved in QSS algorithms using \emph{logarithmic quantization} \citep{kofman2009relative}, where the quantum becomes larger as the absolute value of the state $|x_i(t)|$ grows. 

For this purpose, the quantum is computed using two parameters $\dQrel$ and $\dQabs$ where the first one is multiplied by the absolute value of the state and the second one defines the minimum value the quantum can take (for the case in which the state has a value near zero). These parameters have same role than the relative and absolute tolerance of classic variable step size ODE solvers.  

\subsection{The Stand-Alone QSS Solver}
The Stand-Alone QSS solver \citep{fernandez2014stand}  currently constitutes the most complete and efficient implementation of the QSS family of algorithms. In this environment, models are specified using a restricted subset of the Modelica language \citep{mattsson1998physical}, known as $\mu$-Modelica. The tool automatically translates these specifications into C language encoding the system of ODEs together with the corresponding zero-crossing functions and event-handling routines required for discontinuous dynamics.

Additionally, the tool automatically extracts structural information such as incidence matrices and generates symbolic routines for evaluating the Jacobian matrix. The resulting C code can then be linked either with QSS and LIQSS algorithms of orders one through three, or with well-established conventional solvers such as DOPRI \citep{dormand1980family}, DASSL \citep{petzold1982description}, CVODE \citep{cohen1996cvode}, and IDA \citep{hindmarsh2005sundials}.

An important advantage of this solver is that it provides complete structural information including symbolic sparse Jacobian evaluation, enabling conventional integrators to achieve significantly faster performance compared to other interfaces \citep{kofman2021compact}. Moreover, since it generates the same piece of C code for the model equations used across both QSS-based and classical solvers, the framework offers a rigorous basis for fair performance comparisons between them.

\subsection{Quantum Size and Number of Steps in QSS methods}
% The concept of \emph{activity of order $n$} allows estimating a lower bound for the number of steps that a QSS method of order $n$ will require to complete a simulation with a given tolerance.
The concept of activity was originally proposed as a way to characterize the rate of variation of continuous-time signals \citep{jammalamadaka2003activity}.  
For a signal $x_i(t)$ with integrable derivative, its activity over a time window $[t_0,t_f]$ is given by:

\begin{equation} \label{eq:act1}
    A_{x_i(t_0,t_f)}^{(1)} \triangleq \int_{t_0}^{t_f} \left |\cfrac{dx_i(\tau)}{d \tau} \right| \, d \tau.
\end{equation}

This quantity represents the accumulated variation of the signal within the interval, effectively measuring the total displacement across successive extrema.  
In the particular case of monotonic trajectories, Eq.~\eqref{eq:act1} reduces to the difference between the maximum and minimum values.  
As such, activity provides a direct link with the number of intervals required in a piecewise-constant representation of the trajectory, as illustrated in Figure~\ref{fig:qss}(a). 

This notion was extended in \citep{castro2015activity}, introducing the concept of $n$-th order activity that measures the variation of the higher-order derivatives of the signal. This generalization allows for the analysis of approximations based on polynomials of degree up to $n-1$, such as those illustrated in Figure~\ref{fig:qss}(b). Formally, the $n$-th order local activity of a signal $x_i(t)$ on $[t_0,t_f]$ is defined as:

\begin{equation}\label{eq:act_n}
    A_{x_i(t_0,t_f)}^{(n)} \triangleq \bigintss_{t_0}^{t_f} \left | \frac{1}{n!} \cdot \cfrac{d^n x_i(\tau)}{d \tau^n} \right|^{1/n} \, d\tau 
\end{equation}
The concept of $n$--th order activity was used to estimate the minimum number of steps required by QSS and LIQSS algorithms depending on the quantum size.  This lower bound, for a system like that of Eq.\eqref{eq:ode}, is given by the expression

\begin{equation}\label{eq:k_qssOld}
 k_{\mathrm{QSSn}(t_0,t_f)}(\dQ) = \sum_{i=1}^N \cfrac{\lacts{n}}{ (\Delta Q_i)^{\frac{1}{n}}} 
\end{equation}
where $\dQ=[\Delta Q_1,\ldots,\Delta Q_N]^T$. This bound assumes that each segment of $q_i$ starts (or ends) with the same value than $x_i$. Thus, all existing QSS and LIQSS algorithms perform a number of steps that is lower bounded by Eq.\eqref{eq:k_qssOld}.

Recently, in \citep{bergonzi2025activity}, the ideas of activity were further extended and it was shown that the minimum number of steps in any type of quantization-based algorithm (including hypothetical generalized QSS and LIQSS algorithms) are lower bounded by a smaller quantity given by

\begin{equation}\label{eq:k_qss}
 \kQSS{n}(\dQ) = \sum_{i=1}^N \cfrac{\lacts{n}}{2^{\frac{2n-1}{n}}\cdot (\Delta Q_i)^{\frac{1}{n}}} 
\end{equation}

One of the main goals of this work is to develop new QSS algorithms (of LIQSS type) where the number of steps can get closer to the lower bound of Eq.\eqref{eq:k_qss} rather than to that of Eq.\eqref{eq:k_qssOld}. In these new methods, each segment of $q_i$ will not start or end with the same value of $x_i$.

% In addition, incorporating the definitions of \emph{instantaneous},  \emph{local} and \emph{global activities}, that work also establishes a lower bound for the number of steps in classic variable-step solvers depending on the error tolerance. 

% This expression makes it possible to estimate the minimum theoretical number of polynomial segments of degree at most $n-1$ required to approximate the signal $x_i(t)$ within a prescribed error tolerance $\Delta Q_i$ over the interval $[t_0,t_f]$. Then, it is possible to establish a lower bound on the number of steps required by a quantized state method to simulate a system such as the one defined in Eq.~\eqref{eq:ode} over the time interval $(t_0,t_f)$. This theoretical minimum number of steps can be estimated as

\section{Generalization of LIQSS Methods} \label{sec:main}
In this section, we present a procedure to desing general LIQSS methods based  on the definition of a polynomial $p_i(t)$ that represents the difference between $q_i(t)$ and $x_i(t)$. We first show that with the appropriate choice of this polynomial we can retrieve existing LIQSS algorithms and that they can be extended to perform larger steps. Then, we show a way of defining $p_i(t)$ such that the step size is maximized arriving to a new family of algorithms called Chebyshev LIQSS method.

\subsection{Basic Formulation}
Consider the system of Eq.\eqref{eq:ode}, the QSS approximation of Eq.\eqref{eq:qss} and the linear approximation of its components given by Eqs.\eqref{eq:dxi}--\eqref{eq:linearap}. 
We will assume, without lost of generality, that $t^i_j=0$. Also, in order to simplify the notation, we shall simply write $a_i(t^i_j)=a_i(0)=a_i$.  

We define then the difference polynomial $p_i(t)$ as
\begin{equation} \label{eq:pi}
 p_i(t) = x_i(t)-q_i(t)
\end{equation}
and the different LIQSS methods will be designed such that they verify Eq.\eqref{eq:pi} for a proposed polynomial. This is, rather than starting from a definition for the quantized state trajectory $q_i(t)$, we will provide a definition for $p_i(t)$.

That way, given the desired difference polynomial $p_i(t)$ of order $n$, the problem is to obtain the quantized state trajectory $q_i(t)$ of order $n-1$ such that the resulting state trajectory $x_i(t)$ of order $n$ obtained after integrating Eq.\eqref{eq:dxi} effectively verifies Eq.\eqref{eq:pi}.  

In order to obtain the $n$ coefficients that define each segment of the polynomial $q_i(t)$, a system of equations can be derived by solving Eq.\eqref{eq:pi} for $q_i$, differentiating, and replacing $\dot x_i(t)$ in the linear approximation of Eq.\eqref{eq:dxi}:
% \begin{equation}\label{eq:pip}
% \dot p_i(t) = \dot x_i(t)- \dot q_i(t)= a_i q_i(t)+  u_i(t) - \dot q_i(t)
% \end{equation}
% By differentiating $\dot p_i(t)$ and the successive derivatives, and recalling the definition of $p_i(t)$ in Eq.\eqref{eq:pi} we arrive to the following system of equations 
\begin{subequations}\label{eq:pq}
\begin{align}
q_i(t)&=x_i(t) - p_i(t)\label{eq:pq0}\\
\dot q_i(t)&= a_i q_i(t)+  u_i(t)-\dot p_i(t) \label{eq:pq1}\\
\ddot q_i(t)&=  a_i \dot q_i(t)+  \dot u_i(t) - \ddot p_i(t)\label{eq:pq2}\\
\dddot q_i(t)&= a_i \ddot q_i(t)+  \ddot u_i(t)-\dddot p_i(t) \label{eq:pq3} 
\end{align}
\end{subequations}
where we limited the derivatives of $q_i(t)$ up to order three since we will not develop algorithms of higher order.

Let us define the auxiliary coefficients
\begin{subequations} \label{eq:ri}
    \begin{align}
        r_1&\triangleq a_i x_i(0)+u_i(0)\\
      r_2&\triangleq a_i^2 x_i(0)+a_i u_i(0)+\dot u_i(0) \\
       r_3 &\triangleq a_i^3 x_i(0)+a_i^2 u_i(0)+a_i \dot u_i(0)+\ddot u_i(0)        
    \end{align}
\end{subequations}
and 
\begin{subequations} \label{eq:si}
    \begin{align}
      s_1&\triangleq a_i p_i(0)+\dot p_i(0)\\   
      s_2&\triangleq a_i^2 p_i(0)+a_i \dot p_i(0) +\ddot p_i(0)\\
      s_3&\triangleq a_i^3 p_i(0)+a_i^2 \dot p_i(0)+a_i \ddot p_i(0)+\dddot p_i(0)  
    \end{align}
\end{subequations}
Notice that the coefficients $r_k$ and $s_k$ only depend on the  initial state $x_i(0)$, the polynomial $p_i(t)$ and the linear approximation parameters. Then, replacing Eq.~\eqref{eq:pq0} in Eq.~\eqref{eq:pq1}  we obtain
\begin{subequations} 
    \begin{align*}
        \dot q_i(0)&= a_i (x_i(0)-p_i(0)+  u_i(0)-\dot p_i(0)\\
        &= a_i x_i(0)+u_i(0) - a_i p_i(0)-\dot p_i(0)\\
        &= r_1 - s_1
    \end{align*}
\end{subequations}
Repeating the same procedure for $\ddot q_i(t) $ and $\dddot q_i(t)$ we obtain:
\begin{subequations} \label{eq:q1q2q3}
    \begin{align}
    q_i(0) &= x_i(0)-p_i(0)\\
   \dot q_i(0) &= r_1 - s_1\\
   \ddot q_i(0)&=  r_2 - s_2\\
   \dddot q_i(0)&= r_3 - s_3    
    \end{align}
\end{subequations}
This provides a system of equations for finding the coefficients of the polynomial $q_i(t)$ depending on the difference polynomial $p_i(t)$, and also depending on the linear approximation represented by $a_i$ and $u_i(t)$.

Following the principles of QSS methods, $p_i(t)$ must be such that $|p_i(t)|\leq \Delta Q_i$ in the interval $[0,t^i_1]$ (where $t^i_1$ is the time of the next change in the quantized state $q_i$). This leads to two possible situations:

\begin{enumerate}
 \item Like in all LIQSS methods, whenever is possible $p_i(t)$ must be selected as a constant polynomial so that $q_i(t)$ and $x_i(t)$ run parallel to each other and no further steps are required. This \emph{equilibrium} situation is characterized by a polynomial $p_i(t)=c$ with $|c|\leq \Delta Q_i$.  In this case, it can be easily seen from Eq.\eqref{eq:si} that
\begin{equation} \label{eq:sn}
   s_n=a_i^n c
\end{equation}
 and the system of Eq.\eqref{eq:q1q2q3} results 
\begin{subequations}\label{eq:pqeq}
\begin{align}
q_i(0)&=x_i(0) - c \\
\dot q_i(0)&=  r_1-a_i c\\
\ddot q_i(0)&= r_2 -a_i^2 c\\
\dddot q_i(0)&= r_3 - a_i^3 c
\end{align}
\end{subequations}

Recalling that in the $n$-th order algorithm $\frac{\diff ^n q_i}{\diff t^n}=0$, it results $r_n = s_n = a_i^n c$, from which, assuming that $a_i\neq 0$, we have
\begin{equation} \label{eq:c}
   c = \frac{r_n}{a_i^n}  
\end{equation}
 where the condition $|c|\leq \Delta Q_i$ must be verified. Thus, the equilibrium solution for a method of order $n$ exists if and only if
\begin{equation} \label{eq:equil}
   \left |\frac{r_n}{a_i^n} \right | \leq \Delta Q_i 
\end{equation}
 when $a_i\neq 0$. 
 
 In case $a_i=0$, according to Eq.\eqref{eq:pq}, the equilibrium situation requires that $\frac{\diff ^{n-1} u_i}{\diff t^{n-1}}=\frac{\diff ^n q_i}{\diff t^n}=0$, which in turn implies that $r_n=0$. In that case, we can take $q_i(t)=x_i(t)$ since the equilibrium condition is verified for any value of $c$. 
 
 Then, the quantized state trajectories for the equilibrium condition can be characterized as follows:

\begin{subequations}\label{eq:qequil}
\begin{align}
q_i(t)&=
\begin{cases}
    x_i(t) - \cfrac{r_n}{a_i^n}&\text{ if }|r_n|\leq |a_i|^n \Delta Q_i \wedge a_i\neq 0\\
    x_i(t) &\text{ if } r_n=a_i= 0
\end{cases}
\\
\dot q_i(t)&= a_i q_i(t)+  u_i(t) \\
\ddot q_i(t)&=  a_i \dot q_i(t)+  \dot u_i(t)
\end{align}
\end{subequations}
where $r_n$ is defined in Eq.\eqref{eq:ri}

 \item  In case the equilibrium solution for $p_i(t)=c$ with $|c|\leq \Delta Q_i$ does not exists, we will define $q_i(t)$ and the corresponding derivatives (depending on the order of the method) based on the choise of $p_i(t)$.
\end{enumerate}

Taking into account these considerations, we shall develop next the LIQSS and the Chebyshev LIQSS algorithms of order 1 to 3. 

\subsection{Derivation of LIQSS Methods}
In Linearly Implicit QSS methods each quantized state $q_i(t)$ is computed so that it moves towards the state $x_i(t)$ starting from a distance $\Delta Q_i$. Using the definition of Eq.\eqref{eq:pi}, a LIQSS method of order $n$ can be formulated taking
\begin{equation}\label{eq:pliqss}
    p_i(t)=b_n (t_m-t)^n
\end{equation}
where $t_m>0$ is the instant of time at which $q(t)=x(t)$ and $b_n$ is the leading term of $x_i(t)$ (which is unknown, since it depends on the choice of the quantized state). 

The time derivatives of $p_i(t)$ are
\begin{align*}
    \dot p_i(t)&=-n b_n(t_m-t)^{n-1} \\
    \ddot p_i(t)&=n(n-1) b_n(t_m-t)^{n-2} \\
        \dddot p_i(t)&=-n(n-1)(n-2) b_n(t_m-t)^{n-3} 
\end{align*}
From these equations, we can rewrite
\begin{align*}
     p_i(0)&=b_n t_m^n\\
    \dot p_i(0)&=-n\frac{p_i(0)}{t_m}\\
   \ddot p_i(0)&=n(n-1) \frac{p_i(0)}{t_m^2}\\
   \dddot p_i(0)&=-n(n-1)(n-2) \frac{p_i(0)}{t_m^3}
\end{align*}
Then, from Eq.\eqref{eq:si} it results
\begin{subequations}
\begin{align}
    s_1&=a_i p_i(0)+\dot p_i(0) = a_i p_i(0) - n\frac{p_i(0)}{t_m} \label{eq:s1}\\
    s_2&=a_i s_1+\ddot p_i(0) = a_i^2 p_i(0)- a_i n\frac{p_i(0)}{t_m}+n(n-1) \frac{p_i(0)}{t_m^2}\label{eq:s2}\\
    s_3&=a_i s_2+\dddot p_i(0) = a_i^3 p_i(0)- a_i^2 n\frac{p_i(0)}{t_m}+ \nonumber\\
    &+a_i n(n-1) \frac{p_i(0)}{t_m^2}- n(n-1)(n-2)\frac{p_i(0)}{t_m^3} \label{eq:s3}
\end{align}
\end{subequations}

Evaluating $p_i$ at $t=0$, the condition 
\begin{equation}\label{eq:pliqss0}
    |p_i(0)|=|b_n|\cdot t_m^n = \Delta Q_i
\end{equation}
must be verified except for the equilibrium case of Eq.\eqref{eq:qequil}. 

\subsubsection{First order LIQSS method}
In the LIQSS1 method, taking $n=1$ in Eq.\eqref{eq:s1} and recalling that $s_1=r_1$, it results
\begin{equation*}
    s_1= a_i p_i(0) - \frac{p_i(0)}{\tm(1)} =r_1
\end{equation*}
thus,
\begin{equation*}
    \left (\frac{r_1}{p_i(0)}-a_i\right)\tm(1)+1=0
\end{equation*}
Recalling that $|p_i(0)|=\Delta Q_i$ and that $|r_i/a_i|\geq \Delta Q_i$ (otherwise the equilibrium condition of Eq.\eqref{eq:equil} would be verified), then the sign of $p_i(0)$ must be opposite to that of $r_1$ (otherwise $t_m$ would result negative). Then, it results
\begin{equation}
    p_i(0)=-\sign(r_1)\Delta Q_i
\end{equation}
and the equations for $q_i$, taking also into account the equilibrium case of Eq.~\eqref{eq:qequil}, are
\begin{equation}\label{eq:liqss1}
    q_i=
\begin{cases}
    x_i-\cfrac{r_1}{a_i}&\text{ if } |r_1|\leq |a_i|\Delta Q_i \wedge a_i\neq 0\\
    x_i&\text{ if } r_1=a_1= 0\\
    x_i+\sign(r_1)\Delta Q_i&\text{ otherwise }
\end{cases}
\end{equation}
with $r_1=a_i x_i+u_i$. 

\subsubsection{Second order LIQSS method}
Taking $n=2$ in Eqs.\eqref{eq:s2} and recalling that $s_2=r_2$, it results
\begin{equation}
    s_2=  a_i^2 p_i(0)- 2 a_i \frac{p_i(0)}{t_m}+2 \frac{p_i(0)}{t_m^2}=r_2
\end{equation}
thus,
\begin{equation} \label{eq:tm2}
    \left (\frac{r_2}{p_i(0)}-a_i^2\right)t_m^2+2 a_i t_m-2=0
\end{equation}
For this equation to have a real positive solution for $t_m$ requires that the leading term is positive, which is only achieved when $r_2$ and $p_i(0)$ have the same sign. Thus, 
\begin{equation}
    p_i(0)=\sign(r_2)\Delta Q_i
\end{equation}
and the values for $q_i$ and $\dot q_i$, taking also into account the equilibrium case of Eq.\eqref{eq:qequil}, are
\begin{subequations}\label{eq:liqss2}
\begin{align}
q_i&=
\begin{cases}
    x_i - \cfrac{r_2}{a_i^2}&\text{ if }|r_2|\leq |a_i|^2 \Delta Q_i \wedge a_i\neq 0\\
    x_i &\text{ if } r_2=a_i= 0 \\
     x_i - \sign(r_2) \Delta Q_i & \text{ otherwise}
\end{cases}
\\
\dot q_i&=
\begin{cases}
    a_i q_i+  u_i&\text{ if }|r_2|\leq |a_i|^2 \Delta Q_i \\
    a_i q_i + u_i + 2 \cfrac{\sign(r_2)\Delta Q_i}{t_m} & \text{ otherwise}
\end{cases}
\end{align}
\end{subequations}

% \begin{equation} \label{eq:liqss2}
% \begin{split}
%  &\begin{cases}
%  q_i(t)&= x_i-\cfrac{r_2}{a_i^2}\\
%  \dot q_i(t)&=a_i q_i + u_i  
%  \end{cases}
%  \quad \text{if }|r_2|\leq a_i^2\dQ_i \wedge a_i\neq 0\\
% &\begin{cases}
% q_i(t)&= x_i - \sign(r_2) \Delta Q_i\\
% \dot q_i(t)&= a_i q_i + u_i + 2 \cfrac{\sign(r_2)\dQ_i}{t_m}\\
% \end{cases}
% \quad \text{otherwise }
% \end{split}
% \end{equation}
where $t_m$ is the only positive solution of Eq.\eqref{eq:tm2} with  $r_2=a_i^2 x_i +a_i u_i + \dot u_i$. 

\subsubsection{Third order LIQSS method}
Taking $n=3$ in Eqs.\eqref{eq:s3} and recalling that $s_3=r_3$, it results
\begin{equation} 
    s_3=  a_i^3 p_i(0)- 3 a_i^2 \frac{p_i(0)}{t_m}+6 a_i\frac{p_i(0)}{t_m^2}-6\frac{p_i(0)}{t_m^3}=r_3
\end{equation}
thus,
\begin{equation} \label{eq:tm3}
    \left (\frac{r_3}{p_i(0)}-a_i^3\right)t_m^3+3 a_i^2 t_m^2-6 a_i t_m +6=0
\end{equation}
This equation has solution $t_m>0$ for every $a_i$ provided that the leading term is negative. Thus, the sign of $p_i(0)$ must be opposite to that of $r_3$, resulting
\begin{equation}
    p_i(0)=-\sign(r_3)\Delta Q_i
\end{equation}
and the values for $q_i$, $\dot q_i$ and $\ddot q_i$,
taking also into account the equilibrium case of Eq.\eqref{eq:qequil}, are
\begin{subequations}\label{eq:liqss3}
\begin{align}
q_i&=
\begin{cases}
    x_i - \cfrac{r_3}{a_i^3}&\text{ if }|r_3|\leq |a_i|^3 \Delta Q_i \wedge a_i\neq 0\\
    x_i &\text{ if } r_3=a_i= 0 \\
     x_i + \sign(r_3) \Delta Q_i & \text{ otherwise}
\end{cases}
\\
\dot q_i&=
\begin{cases}
    a_i q_i+  u_i&\text{ if }|r_3|\leq |a_i|^3 \Delta Q_i \\
    a_i q_i + u_i + 3 \cfrac{\sign(r_3)\Delta Q_i}{t_m} & \text{ otherwise}
\end{cases}
\\
\ddot q_i&=
\begin{cases}
    a_i \dot q_i+  \dot u_i&\text{ if }|r_3|\leq |a_i|^3 \Delta Q_i \\
    a_i \dot q_i + \dot u_i - 6 \cfrac{\sign(r_3)\Delta Q_i}{t_m^2} & \text{ otherwise}
\end{cases}
\end{align}
\end{subequations}

% \begin{equation}  \label{eq:liqss3}
% \begin{split}
%  &\begin{cases}
%  q_i(t)&= x_i-\cfrac{r_3}{a_i^3}\\
%  \dot q_i(t)&=r_1-a_i\cfrac{r_3}{a_i^3}\\ 
%  \ddot q_i(t)&=r_2-a_i^2\cfrac{r_3}{a_i^3} 
%  \end{cases}
%  \quad \text{if }|r_3|\leq |a_i|^3\dQ_i \wedge a_i\neq 0\\
% &\begin{cases}
% q_i(t)&= x_i + \sign(r_3) \Delta Q_i\\
% \dot q_i(t)&= a_i q_i + u_i + 3 \cfrac{\sign(r_3)\dQ_i}{t_m}\\
% \ddot q_i(t)&= a_i \dot q_i + \dot u_i - 6 \cfrac{\sign(r_3)\dQ_i}{t_m^2}\\
% \end{cases}
% \quad \text{otherwise }
% \end{split}
% \end{equation}
where $t_m$ is the only real valued positive solution of Eq.\eqref{eq:tm3} with  $r_3=a_i^3 x_i + a_i^2 u_i + a_i \dot u_i + \ddot u_i$. 

\subsubsection{Quantized State Update Policy}
Equations \eqref{eq:liqss1}, \eqref{eq:liqss2} and \eqref{eq:liqss3} provide the formulas to compute the quantized states in first, second and third order LIQSS methods. However, they do not tell how to compute the time of the next change in the quantized state. Regarding that, there are two alternatives:
\begin{itemize}
    \item In the original formulation  of \citep{migoni2013linearly} LIQSS steps are performed immediately after the quantized state $q_i$ reaches the state $x_i$. In case the state trajectory changes its direction (due to a change in another quantized state) and it starts diverging from $q_i$, then the step is performed when the difference is equal to $2 \Delta Q_i$ (for this reason, LIQSS algorithms had twice the error bound).

    In order to avoid this larger error bound, in this work we propose to perform a step in that case when the difference reaches $\Delta Q_i$ instead of $2\Delta Q_i$.

    \item An alternative to this solution is to perform the steps only when the difference between the state and the quantized state becomes the quantum (irrespective of the fact that $q_i(t)$ reaches $x_i(t)$). That way, larger steps can be performed without increasing the error bound since the difference between $x_i$ and $q_i$ will be bounded by $\Delta Q_i$ anyway. 

    We will call \emph{extended} LIQSS (eLIQSS) to the algorithms that  use this second update policy.
\end{itemize}

\subsection{Chebyshev LIQSS Method}
The central idea of the Chebyshev LIQSS (\CQSS) methods is to use a difference polynomial $p_i(t)$ that maximizes the time interval during which it remains bounded by $\pm\Delta Q_i$. This idea  maximizes the time during which $q_i(t)$ remains close to $x_i(t)$ which maximizes in turn the interval between successive steps, and reduces (at least in theory) the total number of steps. 

\subsubsection{Step size maximization}\label{sec:step_max}
The basis of \CQSS methods is provided by the following result
\begin{theorem}[\citep{bergonzi2025activity}]
Given a signal $x_i(t)$ expressed as a polynomial of degree $n$, there exists a
polynomial $q_i(t)$ of degree less than or equal to $n-1$ satisfying the condition $|q_i(t)-x_i(t)| \leq \Delta Q_i$
in an interval $[t_j, t_j +\Delta t]$ if and only if 
\begin{equation}\label{eq:deltat}
   \Delta t\leq \tm(n)\triangleq \frac{2^{\frac{2n-1}{n}} \Delta Q_i^{\frac{1}{n}}}{\ilact{n}(t_j)} 
 \end{equation}
where 
\begin{equation}\label{eq:inst_act}
    \ilact{n}(t)\triangleq\left | \frac{1}{n!} \cdot \cfrac{\diff^n x_i(t)}{\diff t^n} \right|^{1/n}
\end{equation} 
is the local instantaneous activity of order $n$ of the signal $x_i(t)$.
 \end{theorem}

That result establishes an upper bound for the time that a state trajectory $x_i(t)$ and a quantized state trajectory $q_i(t)$ can remain close to each other.  Then, the following theorem constructs a polynomial $q_i(t)$ of order $n-1$ (or less) that verifies that upper bound.

\begin{theorem}\label{th:q_i}
    Let $x_i(t)$ be a polynomial of degree $n$ with leading term $b_n$. Consider the 
 polynomial 
 \begin{equation}\label{eq:q_i_definition}
     q_i(t) = x_i(t) - \sign(b_n)\cdot \Delta Q_i \cdot T_n(\frac{2t-\tm(n)}{\tm(n)})
 \end{equation}
where $T_n(\cdot)$ denotes the $n$-th degree Chebyshev polynomial. Then, $q_i(t)$ has the following properties:
\begin{enumerate}
    \item It has degree equal or less than $n-1$.
    \item It verifies the condition $|q_i(t)-x_i(t)| \leq \Delta Q_i$ in the interval $[0, \tm(n)]$ with $\tm(n)$ defined in Eq.\eqref{eq:deltat}. 
\end{enumerate}
\end{theorem}

\begin{proof} ~

\begin{enumerate}
    \item The leading term of a Chebyshev polynomial $T_n(t)$ is $2^{n-1}$. Then, the leading term of 
    \begin{equation*}
        T_n(\frac{2t-\tm(n)}{\tm(n)})
    \end{equation*}
    results 
    \begin{equation*}
        c_n=\frac{ 2^{n-1} 2^n }{\tm(n)^n}=\cfrac{ 2^{n-1} 2^n }{\left(\cfrac{2^{\frac{2n-1}{n}} \Delta Q_i^{\frac{1}{n}}}{\ilact{n}(0)}\right)^n }=\frac{(\ilact{n}(0))^n}{\Delta Q_i} = \frac{ \left | \cfrac{1}{n!} \cdot \cfrac{\diff^n x_i(t)}{\diff t^n} \right|}{\Delta Q_i}=\frac{|b_n|}{\Delta Q_i}
    \end{equation*}
Then, the leading term of the polynomial
    \begin{equation*}
        p_i(t)=\sign(b_n)\cdot \Delta Q_i \cdot T_n(\frac{2t-\tm(n)}{\tm(n)})
    \end{equation*}
    is $d_n=\sign(b_n)\cdot \Delta Q_i c_n=\sign(b_n)|b_n|=b_n$ showing that $x_i(t)$ and $p_i(t)$ have the same leading term and then $q_i(t)=x_i(t)-p_i(t)$ has degree equal or less than $n-1$.
    
\item Notice that
\begin{equation*}
   t \in [0,\tm(n)] \implies \frac{2t-\tm(n)}{\tm(n)}\in [-1,1] \implies |T_n(\frac{2t-\tm(n)}{\tm(n)})|\leq 1
\end{equation*}
Then, 

\begin{equation*}
   t \in [0,\tm(n)] \implies |p_i(t)|=\Delta Q_i \cdot |T_n(\frac{2t-\tm(n)}{\tm(n)})|\leq \Delta Q_i
\end{equation*}
\end{enumerate}
completing the proof    
\end{proof}

This last result constitutes the basis for the design of the  Chebyshev LIQSS methods. Using the definition of Eq.~\eqref{eq:pi} and taking into account Eq.~\eqref{eq:q_i_definition}, a \CQSS method of order $n$ can be formulated selecting

 \begin{equation}\label{eq:pi2}
     p_i(t)=\sign(b_n)\cdot \Delta Q_i \cdot T_n(\frac{2t-\tm(n)}{\tm(n)})
 \end{equation}

It is worth mentioning that $x_i(t)$ depends on $q_i(t)$. Thus, nor $b_n$ neither $\tm(n)$ are known in advance. 
\subsubsection{First Order Chebyshev LIQSS Method}

In the first order Chebyshev LIQSS method, taking $n=1$  and  recalling that $T_1(z)=z$, we can obtain from Eq.~\eqref{eq:pi2}  the expression of the polynomial $p_i(t)$ and its first derivative as follows
\begin{subequations}
    \begin{align*}
        p_i(t)&=\sign(b_1)\Delta Q_i \cdot  \frac{2t-\tm(1)}{\tm(1)} \\
        \dot p_i(t)&=\sign(b_1) \Delta Q_i \cdot \frac{2t-\tm(1)}{\tm(1)} \cdot \frac{2}{\tm(1)}
    \end{align*}
\end{subequations}
Then, evaluating these expressions in $t=0$, they result:
\begin{subequations}\label{eq:pi0_cqss} 
    \begin{align}
         p_i(0)&=-\sign(b_1) \Delta Q_i \\
         \dot p_i(0)&=\sign(b_1) \Delta Q_i \frac{2}{\tm(1)} =-p_i(0)\frac{2}{\tm(1)}
    \end{align}
\end{subequations}
and replacing in the expression of $s_1$ of Eq.\eqref{eq:si} we have
\begin{equation*}
           s_1=a_i p_i(0)+\dot p_i(0) = a_i p_i(0)-p_i(0)\frac{2}{\tm(1)} 
\end{equation*}
Recalling that $s_1=r_1$, it results:
 \begin{equation}\label{eq:t1}
     \left(\frac{r_1}{p_i(0)}+a_i\right)\tm(1)+2=0
 \end{equation}
From Eq.~\eqref{eq:pi0_cqss} we can observe that $|p_i(0)|=\Delta Q_i$ and that $|r_i/a_i|\geq \Delta Q_i$ (otherwise the equilibrium condition of Eq.\eqref{eq:equil} would be verified), then the sign of $p_i(0)$ must be opposite to that of $r_1$ (otherwise $\tm(1)$ would result negative). Then, it results
\begin{equation}
    p_i(0)=-\sign(r_1)\Delta Q_i
\end{equation}
and  the equations for $q_i$, taking also into account the equilibrium case of Eq.~\eqref{eq:qequil}, are 
\begin{equation} \label{eq:qiCQSS1}
 q_i=
 \begin{cases}
    x_i-\cfrac{r_1}{a_i}&\text{ if } |r_1|\leq |a_i|\Delta Q_i \wedge a_i\neq 0\\
    x_i&\text{ if } r_1=a_1= 0\\
    x_i+\sign(r_1)\Delta Q_i &\text{ otherwise }
 \end{cases}
\end{equation}
with $r_1=a_i x_i+u_i$. Notice that Eq.\eqref{eq:qiCQSS1} is identical to Eq.\eqref{eq:liqss1} so LIQSS1 and CheQSS1 are equivalent. 

\subsubsection{Second order Chebyshev LIQSS Method}
In the second order algorithm, taking $n=2$
 and recalling that $T_2(z)=2z^2-1$, we can obtain from Eq.~\eqref{eq:pi2} the expression of the polynomial $p_i(t)$ and its derivatives as follows

\begin{subequations}
    \begin{align*}
        p_i(t)&=\sign(b_2)\Delta Q_i \cdot \left[2\left( \frac{2t-\tm(2)}{\tm(2)} \right)^2 -1\right] \\
        \dot p_i(t)&=\sign(b_2) \Delta Q_i \cdot 4 \frac{2t-\tmax(2)}{\tm(2)} \cdot \frac{2}{\tm(2)}\\
        \ddot p_i(t)&=\sign(b_2) \Delta Q_i \cdot \frac{16}{\tm(2)^2} 
    \end{align*}
\end{subequations}

Evaluating the last expressions for $t=0$, it results
\begin{subequations}
    \begin{align*}
         p_i(0)&=\sign(b_2) \Delta Q_i\\
         \dot p_i(0)&=-\sign(b_2) \Delta Q_i \frac{8}{\tm(2)}\\
         \ddot p_i(0)&=\sign(b_2) \Delta Q_i \frac{16}{\tm(2)^2}
         \end{align*}
\end{subequations}

Then, replacing with these expressions in Eq.\eqref{eq:sn}, we obtain 
\begin{equation*}
 \begin{split}
       s_1&=a_i p_i(0)+\dot p_i(0) = a_i p_i(0)-p_i(0)\frac{8}{\tm(2)}\\
       s_2&= a_i s_1+\ddot p_i(0) = a_i^2 p_i(0)- a_i p_i(0)\frac{8}{\tm(2)}+p_i(0)\frac{16}{\tm(2)^2}
 \end{split}        
\end{equation*}
 and recalling that $s_2=r_2$, it results:
 \begin{equation}\label{eq:t2}
     \left(\frac{r_2}{p_i(0)}-a_i^2\right)\tm(2)^2+8a_i \tm(2)-16=0
 \end{equation}

This is a quadratic equation in $\tm(2)$ with discriminant 
\begin{equation*}
    \Delta = 64 a_i^2 + 64\left(\frac{r_2}{p_i(0)} - a_i^2\right)
    = 64 \frac{r_2}{p_i(0)},
\end{equation*}
which is positive only when $r_2$ and $p_i(0)$ have the same sign. Thus, 
\begin{equation*}
    p_i(0)=\sign(r_2) \Delta Q_i 
\end{equation*}
and the values of  $q_i$ and $\dot q_i$, taking also into account the equilibrium case of Eq.~\eqref{eq:qequil}, are

\begin{subequations}\label{eq:cqss2}
\begin{align}
q_i&=
\begin{cases}
    x_i - \cfrac{r_2}{a_i^2}&\text{ if }|r_2|\leq |a_i|^2 \Delta Q_i \wedge a_i\neq 0\\
    x_i &\text{ if } r_2=a_i= 0 \\
     x_i - \sign(r_2) \Delta Q_i & \text{ otherwise}
\end{cases}
\\
\dot q_i&=
\begin{cases}
    a_i q_i+  u_i&\text{ if }|r_2|\leq |a_i|^2 \Delta Q_i \\
   a_i q_i + u_i + 8\cfrac{\sign(r_2) \Delta Q_i }{\tm(2)} & \text{ otherwise}
\end{cases}
\end{align}
\end{subequations}

where $\tm(2)$ is the only positive solution of Eq.\eqref{eq:t2} with  $r_2=a_i r_1 + \dot u_i$.

% \begin{figure}
%     \centering
%     \includegraphics[width=0.5\linewidth]{cqss2.pdf}
%     \caption{Typical state and quantized state CQSS2 trajectories}
%     \label{fig:cqss2_1}
% \end{figure}

\subsubsection{Third order Chebyshev LIQSS Method}
In the third order algorithm, taking $n=3$ and recalling that $T_3(z)=4z^3-3z$, we can obtain from Eq.~\eqref{eq:pi2} the expression of the polynomial $p_i(t)$ and its derivatives as follows:
\begin{subequations}\label{eq:p123}
    \begin{align}
        p_i(t)&=\sign(b_3)\Delta Q_i \cdot \left[4\left( \frac{2t-\tm(3)}{\tm(3)} \right)^3 -3\left( \frac{2t-\tm(3)}{\tm(3)}\right)\right] \\
        \dot p_i(t)&=\sign(b_3) \Delta Q_i \cdot \left[12\left(\frac{2t-\tm(3)}{\tm(3)}\right)^2-3\right] \cdot \frac{2}{\tm(3)}\\
        \ddot p_i(t)&=\sign(b_3) \Delta Q_i \cdot \frac{48}{\tm(3)} \cdot \frac{2t-\tm(3)}{\tm(3)} \cdot \frac{2}{\tm(3)}\\
        \dddot p_i(t)&=\sign(b_3) \Delta Q_i \cdot  \frac{96}{\tm(3)^2} \cdot \frac{2}{\tm(3)}
    \end{align}
\end{subequations}

Evaluating Eqs.~\eqref{eq:p123} in $t=0$, they result:

\begin{subequations}\label{eq:p123_t0}
    \begin{align}
         p_i(0)&=-\sign(b_3) \Delta Q_i\\
         \dot p_i(0)&=\sign(b_3) \Delta Q_i \frac{18}{\tm(3)}\\
         \ddot p_i(0)&=-\sign(b_3) \Delta Q_i \frac{96}{\tm(3)^2}\\
         \dddot p_i(0)&=\sign(b_3) \Delta Q_i \frac{192}{\tm(3)^3}\\
    \end{align}
\end{subequations}

Then, replacing in the expressions of $s_1,s_2$ and $s_3$:
\begin{equation}
 \begin{split}
       s_1&=a_i p_i(0)+\dot p_i(0) = a_i p_i(0)-p_i(0)\frac{18}{\tm(3)}\\
       s_2&= a_i s_1+\ddot p_i(0) = a_i^2 p_i(0)- a_i p_i(0)\frac{18}{\tm(3)}+p_i(0)\frac{96}{\tm(3)^2}\\
       s_3&=a_is_2+\dddot p_i(0)\\
       &=a_i^3p_i(0)-a_i^2p_i(0)\frac{18}{\tm(3)}+a_ip_i(0)\frac{96}{\tm(3)^2}-p_i(0)\frac{192}{\tm(3)^3}
 \end{split}        
\end{equation}
 and recalling now that $s_3=r_3$, it results:
 \begin{equation}\label{eq:t3}
     \left(\frac{r_3}{p_i(0)}-a_i^3\right)\tm(3)^3+18a_i \tm(3)^2-96\tm(3)+192=0
 \end{equation}

% Rewriting Eq.~\eqref{eq:q3p3_2}, and recalling that $r_3= a_i^3 x_i + a_i^2 u_i+a_i \dot u_i + \ddot u_i $ we obtain:

% \begin{align}
%    0&= r_3 + \sign(b_3) \Delta Q_i\left(a_i^3-a_i^2\frac{18}{\tmax(3)}+a_i\frac{96}{\tmax(3)^2}-\frac{192}{\tmax(3)^3}\right) \nonumber \\
%    0&= r_3 + a_i^3 \sign(b_3) \Delta Q_i + \sign(b_3) \Delta Q_i  \left(-a_i^2\frac{18}{\tmax(3)}+a_i\frac{96}{\tmax(3)^2}-\frac{192}{\tmax(3)^3}\right) \nonumber \\
%    0&= \left(\frac{\sign(b_3) r_3}{\Delta Q_i} + a_i^3\right) \tmax(3)^3 - 18 a_i^2 \tmax(3)^2 +96 a_i \tmax(3)-192 \label{eq:t3}
% \end{align}
% \begin{itemize}
%     \item In the case $a_i=0$ the Eq.~\eqref{eq:t3} reduces to: 

% \begin{equation*}
%   \left(\frac{\sign(b_3) r_3}{\Delta Q_i} \right) \tmax(3)^3 -192 =0
% \end{equation*}

% and the sing of $b_3$ must match that of $r_3$. 
% Under this condition, the value of $\tmax(3)$ can be obtained as:
% \begin{equation*}
%     \tmax(3)= \left(192 \frac{\Delta Q_i }{|r_3|}\right)^{1/3}
% \end{equation*}

% \item For the case $a_i<0$, the existence of a positive root requires the leading coefficient of Eq.~\eqref{eq:t3} to be positive, since all the remaining coefficients are negative. This implies that $\frac{\sign(b_3)r_3}{\Delta Q_i}$ must be positive, which again implies that $\sign(b_3)=\sign(r_3)$. 

This equation has always a solution $\tm(3)>0$ provided that the leading term is negative. Thus, the sign of $p_i(0)$ must be opposite to that of $r_3$, resulting 
\begin{equation}
    p_i(0) = -\sign(r_3)\Delta Q_i
\end{equation}
and the values for $q_i$, $\dot q_i$ and $\ddot q_i$, taking also into account the equilibrium case of Eq.~\eqref{eq:qequil}, are

\begin{subequations}
\begin{align}
q_i&=
\begin{cases}
    x_i - \cfrac{r_3}{a_i^3}&\text{ if }|r_3|\leq |a_i|^3 \Delta Q_i \wedge a_i\neq 0\\
    x_i &\text{ if } r_3=a_i= 0 \\
     x_i + \sign(r_3) \Delta Q_i & \text{ otherwise}
\end{cases}
\\
\dot q_i&=
\begin{cases}
    a_i q_i+  u_i&\text{ if }|r_3|\leq |a_i|^3 \Delta Q_i \\
    a_i q_i+  u_i- 18\cfrac{\sign(r_3)\Delta Q_i}{\tm(3)}  & \text{ otherwise}
\end{cases}
\\
\ddot q_i&=
\begin{cases}
    a_i \dot q_i+  \dot u_i&\text{ if }|r_3|\leq |a_i|^3 \Delta Q_i \\
    a_i \dot q_i+ \dot u_i +96 \cfrac{\sign(r_3) \Delta Q_i}{\tm(3)^2} & \text{ otherwise}
\end{cases}
\end{align}
\end{subequations}
where $\tm(3)$ is the only real valued positive solution of Eq.\eqref{eq:t3} with  $r_3=a_i^3 x_i + a_i^2 u_i + a_i \dot u_i + \ddot u_i$. 

\subsection{Implementation of the new methods}
The eLIQSS and \CQSS algorithms were implemented in the Stand-Alone QSS Solver tool \citep{fernandez2014stand}. Since the tool has different functions for defining the quantization policy and the time integration procedure, the implementation of the new algorithms only required to define the functions that compute $q_i(t)$ and its derivatives.

All the new methods are included in the current version version of the tool, which is available at \url{https://github.com/CIFASIS/qss-solver}. 

\subsection{LIQSS, eLIQSS and \CQSS Behavior}
In order to illustrate the behavior of the LIQSS, eLIQSS and \CQSS methods, we simulated a simple first-order ODE given by
\begin{equation}\label{eq:firstOrder}
    \dot{x}_a(t) = 1 - x_a(t)
\end{equation}
with a final time $t_f=5$ and different values for the quantum size. 

Figure~\ref{fig:LIQSS123} shows the trajectories produced by the different algorithms with $\Delta Q=10^{-2}$. Figure~\ref{fig:p_LIQSS123} also shows the trajectory of $p(t)=q(t)-x(t)$ for each method, where it can be appreciated the typical Chebyshev polynomial equi-oscillation properties.

\begin{figure*}[ht!]
 \centering
     \begin{subfigure}[b]{0.32\textwidth}
     \includegraphics[width=\textwidth]{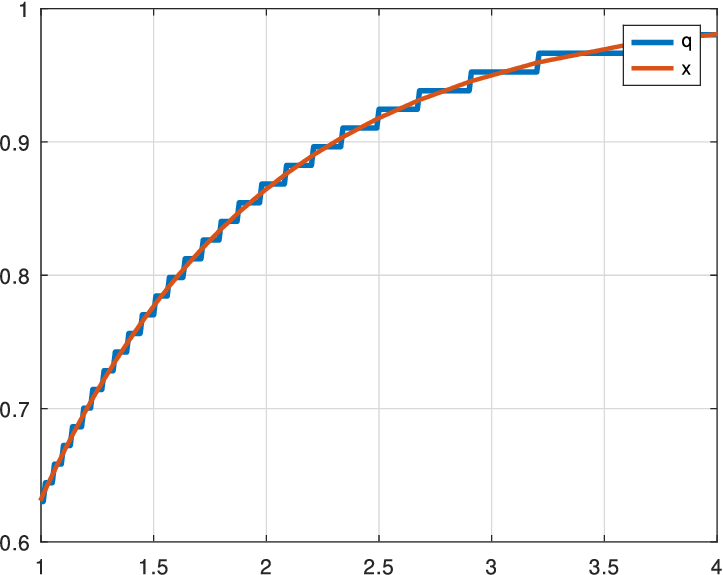}
\caption{ CheQSS1 }
     \label{fig:cqss1}
\end{subfigure} 
 \hfill
\begin{subfigure}[b]{0.32\textwidth}
     \includegraphics[width=\textwidth]{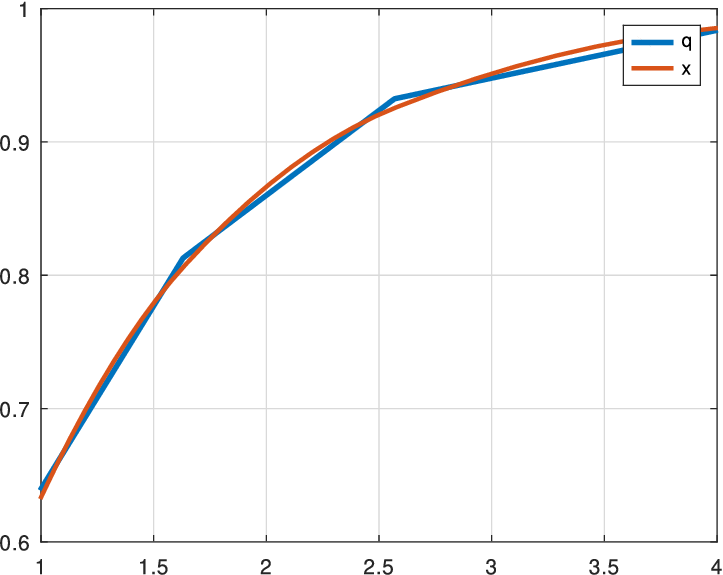}
     \caption{CheQSS2 }
     \label{fig:cqss2}
 \end{subfigure}
  \hfill
\begin{subfigure}[b]{0.32\textwidth}
     \includegraphics[width=\textwidth]{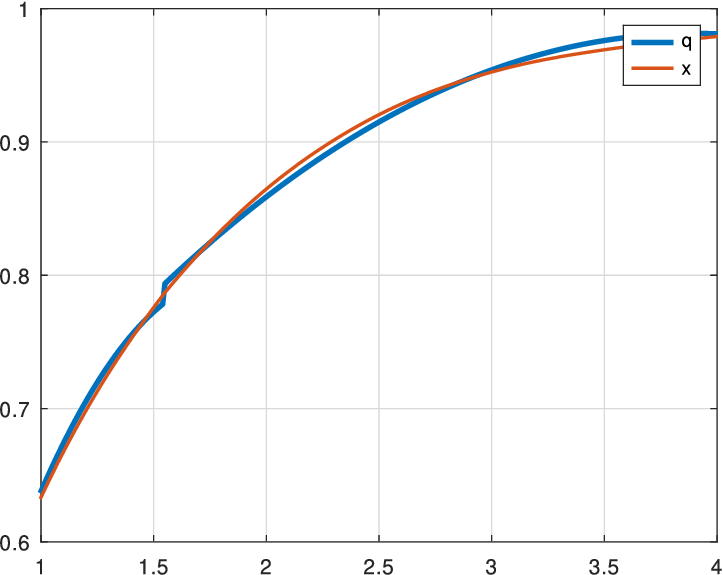}
     \caption{CheQSS3}
     \label{fig:cqss3}
 \end{subfigure}
  \hfill
  \begin{subfigure}[b]{0.32\textwidth}
     \includegraphics[width=\textwidth]{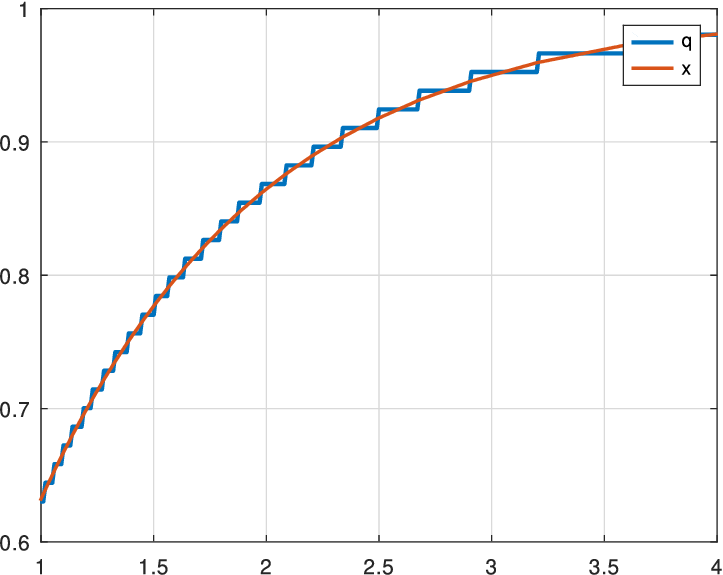}
     \caption{eLIQSS1}
     \label{fig:eliqss1}
     \end{subfigure}
      \hfill
  \begin{subfigure}[b]{0.32\textwidth}
         \includegraphics[width=\textwidth]{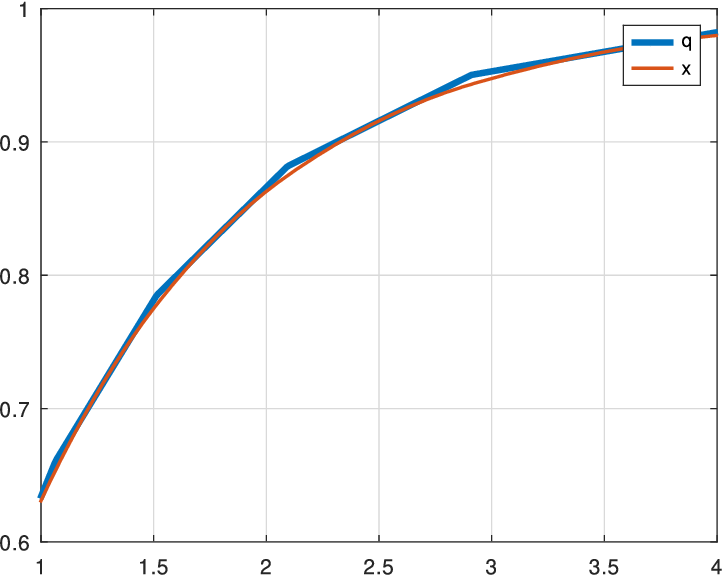}
     \caption{eLIQSS2}
     \label{fig:eliqss2}
 \end{subfigure}
 \hfill
  \begin{subfigure}[b]{0.32\textwidth}
         \includegraphics[width=\textwidth]{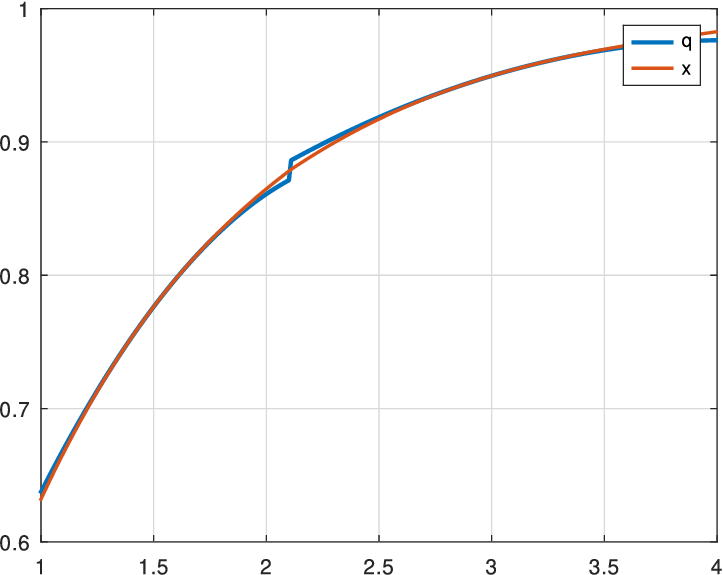}
     \caption{eLIQSS3}
     \label{fig:eliqss3}
 \end{subfigure}
 \hfill
\begin{subfigure}[b]{0.32\textwidth}
     \includegraphics[width=\textwidth]{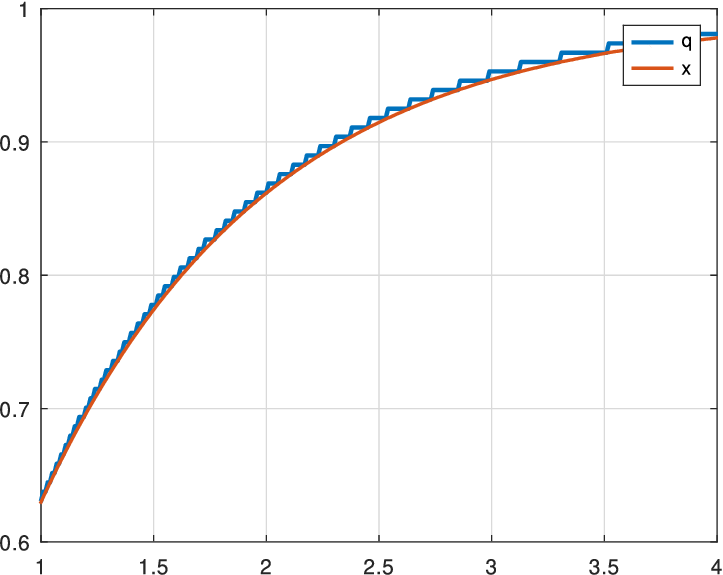}
     \caption{LIQSS1}
     \label{fig:liqss1}
     \end{subfigure}
      \hfill
  \begin{subfigure}[b]{0.32\textwidth}
         \includegraphics[width=\textwidth]{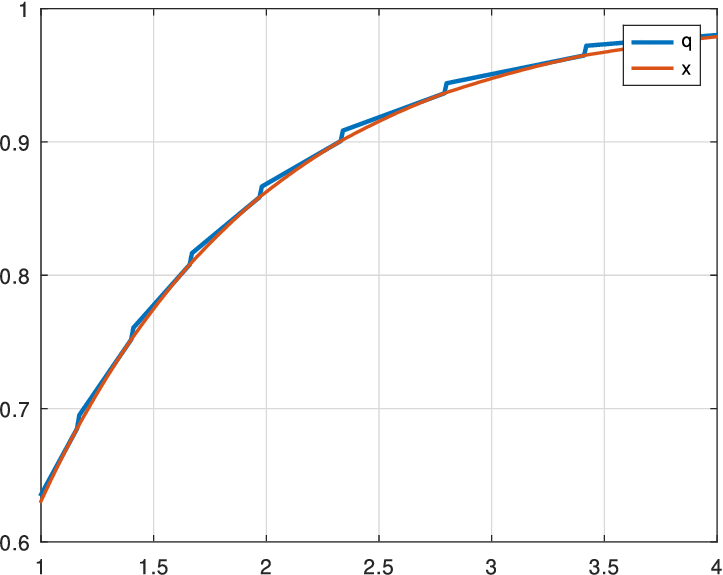}
     \caption{LIQSS2}
     \label{fig:liqss2}
 \end{subfigure}
 \hfill
  \begin{subfigure}[b]{0.32\textwidth}
         \includegraphics[width=\textwidth]{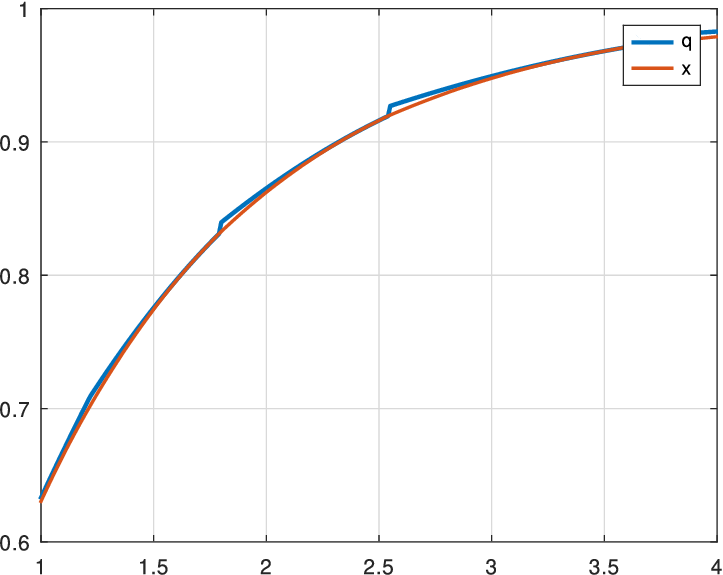}
     \caption{LIQSS3}
     \label{fig:liqss3}
 \end{subfigure}
 \caption{State and Quantized State trajectories of different LIQSS methods (detail) in the simulation of Eq.\eqref{eq:firstOrder}. \label{fig:LIQSS123}}
 \end{figure*}

\begin{figure*}[ht!]
 \centering
     \begin{subfigure}[b]{0.32\textwidth}
     \includegraphics[width=\textwidth]{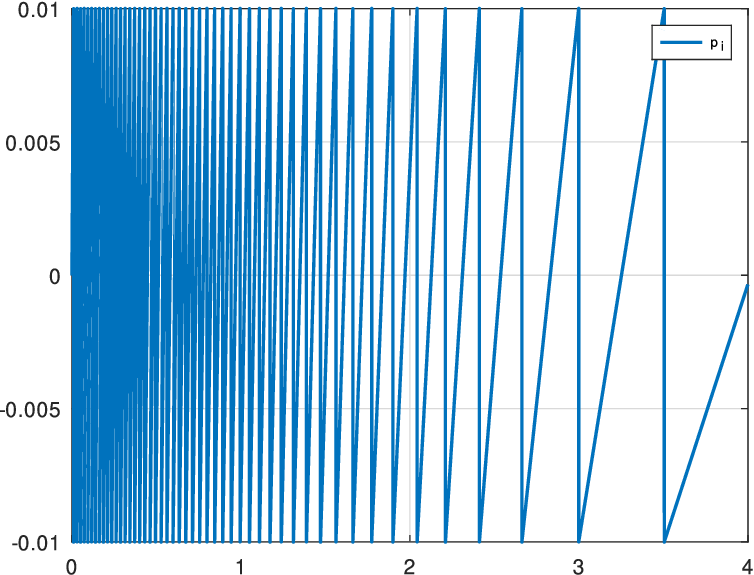}
\caption{ CheQSS1 }
     \label{fig:p_cqss1}
\end{subfigure} 
 \hfill
\begin{subfigure}[b]{0.32\textwidth}
     \includegraphics[width=\textwidth]{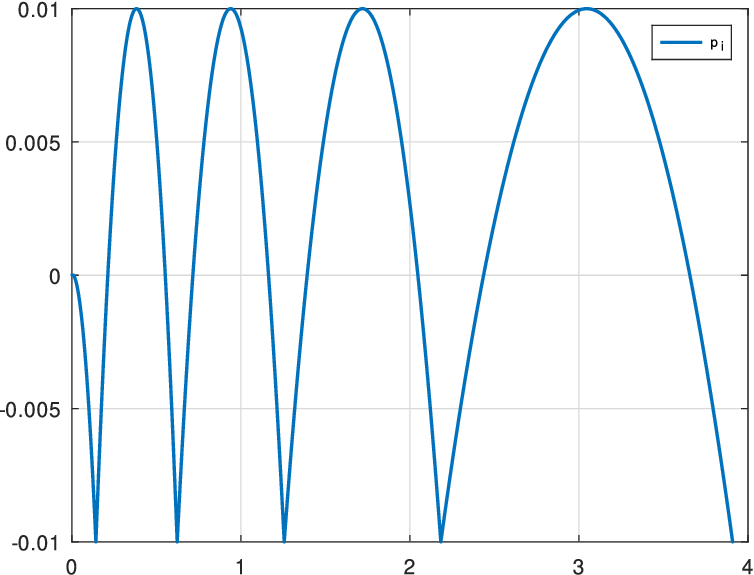}
     \caption{CheQSS2 }
     \label{fig:p_cqss2}
 \end{subfigure}
  \hfill
\begin{subfigure}[b]{0.32\textwidth}
     \includegraphics[width=\textwidth]{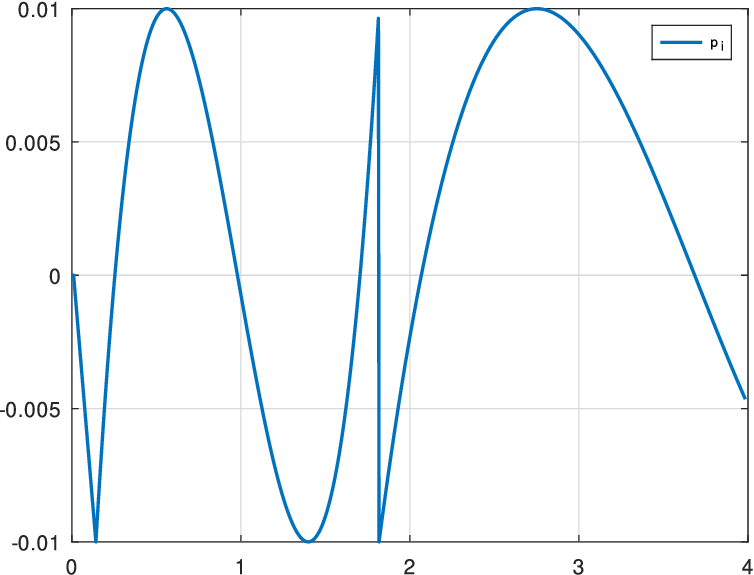}
     \caption{CheQSS3}
     \label{fig:p_cqss3}
 \end{subfigure}
  \hfill
  \begin{subfigure}[b]{0.32\textwidth}
     \includegraphics[width=\textwidth]{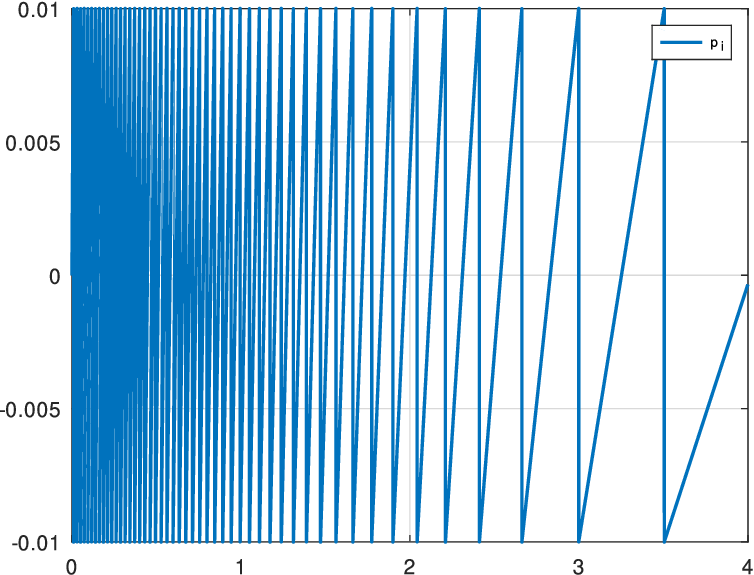}
     \caption{eLIQSS1}
     \label{fig:p_eliqss1}
     \end{subfigure}
      \hfill
  \begin{subfigure}[b]{0.32\textwidth}
         \includegraphics[width=\textwidth]{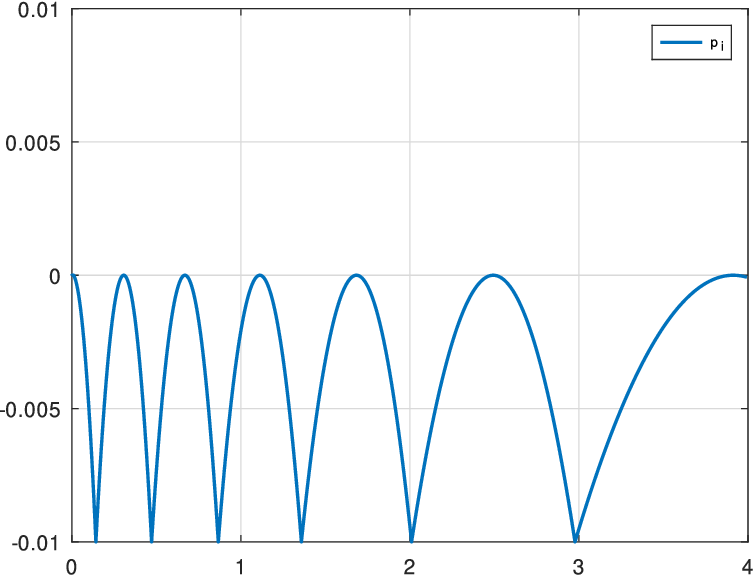}
     \caption{eLIQSS2}
     \label{fig:p_eliqss2}
 \end{subfigure}
 \hfill
  \begin{subfigure}[b]{0.32\textwidth}
         \includegraphics[width=\textwidth]{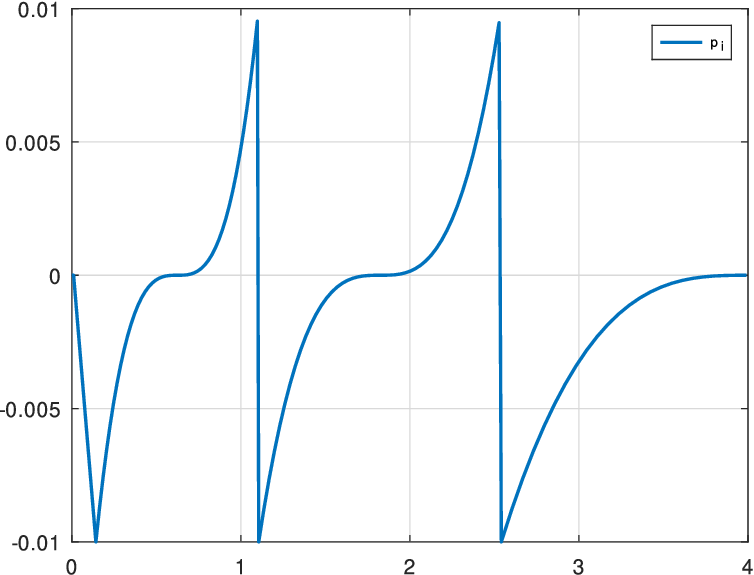}
     \caption{eLIQSS3}
     \label{fig:p_eliqss3}
 \end{subfigure}
 \hfill
\begin{subfigure}[b]{0.32\textwidth}
     \includegraphics[width=\textwidth]{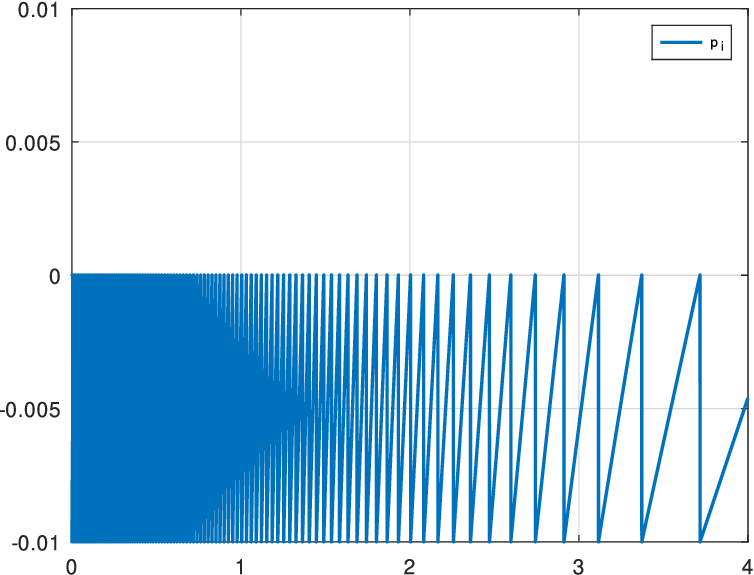}
     \caption{LIQSS1}
     \label{fig:p_liqss1}
     \end{subfigure}
      \hfill
  \begin{subfigure}[b]{0.32\textwidth}
         \includegraphics[width=\textwidth]{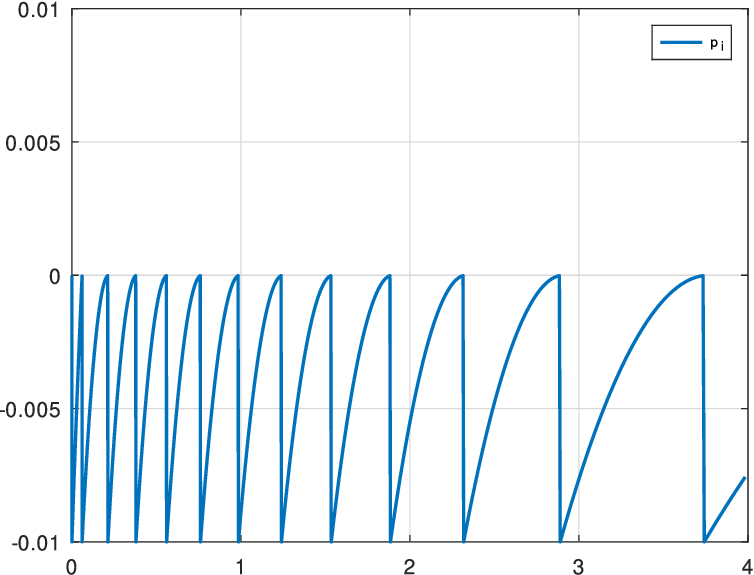}
     \caption{LIQSS2}
     \label{fig:p_liqss2}
 \end{subfigure}
 \hfill
  \begin{subfigure}[b]{0.32\textwidth}
         \includegraphics[width=\textwidth]{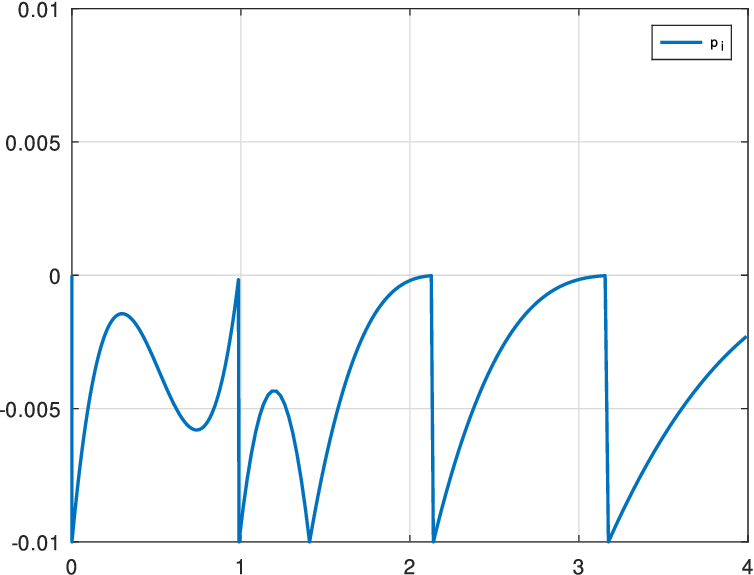}
     \caption{LIQSS3}
     \label{fig:p_liqss3}
 \end{subfigure}
 \caption{Difference between State and Quantized State trajectories of different LIQSS methods (detail) in the simulation of Eq.\eqref{eq:firstOrder}. 
 \label{fig:p_LIQSS123}}
 \end{figure*}

Table~\ref{tab:theoret_vs_real} reports the number of integration steps required by each method under different quantum values, and compares them with the theoretical minimum number of steps computed according to Eq.~\eqref{eq:k_qss}. 

\begin{table}[ht!]
    \caption{Number of steps and their lower bounds in the simulation of Eq.\eqref{eq:firstOrder}.}
    \centering
    \begin{tabular}{@{\extracolsep\fill}cccccc}
    \toprule%
%      \hline
      $n$ & Abs Tol & Theor. Min. & CheQSS$_n$ &eLIQSS$_n$ & LIQSS$_n$  \\
     % \hline
      \midrule
      \multirow{3}{*}{1}  & $10^{-2}$ &50&51	&	51	&	100	\\
       &$10^{-3}$ &497&497	&	497	&	993	\\
       &$10^{-4}$ &4965 &4965	&	4965	&	9924	\\
   %      \hline
   \midrule
       \multirow{3}{*}{2}  & $10^{-2}$ &5&7	&	9	&	15	\\
       &$10^{-3}$ &15&17	&	23	&	44 \\
       &$10^{-4}$ &46&48	&	67	&	136 \\
      %\hline
      \midrule
      \multirow{3}{*}{3}  & $10^{-2}$ &2&4	&	5	&	8\\
       &$10^{-3}$ &5&7	&	9	&	16 \\
       &$10^{-4}$ &10&12	&	17	&	33 \\
        % \hline
\botrule
    \end{tabular}
    \label{tab:theoret_vs_real}
\end{table}

% \begin{figure}[ht!]
%     \centering
%     \includegraphics[width=0.7\linewidth]{tol_vs_steps.eps}
%     \caption{Number of steps performed by different methods varying the tolerance settings.}
%     \label{fig:tol_vs_steps}
% \end{figure}

We can observe that, as expected, since first-order eLIQSS and \CQSS methods are identical to each other, the number of simulation steps they require is also identical and coincide with the theoretical lower bound.  A significant improvement is achieved compared to the original first-order LIQSS method, with approximately a $50\%$ reduction in the number of steps.

When analyzing the number of steps performed by the second and third order algorithms, we see again that both eLIQSS and \CQSS reduce the number of steps with respect to the non extended LIQSS methods. In addition, the number of steps of \CQSS  are very close to the theoretical minimum values.

\subsection{Properties of eLIQSS and \CQSS methods}
The theoretical properties of QSS methods regarding practical stability, convergence and error bounds are derived from the condition $|q_i(t)-x_i(t)|\leq \Delta Q_i$, that implies that the state quantization is equivalent to the addition of bounded disturbances.  Since all eLIQSS and \CQSS algorithms verify the condition $|p_i(t)|=|q_i(t)-x_i(t)|\leq \Delta Q_i$ (see Figure \ref{fig:p_LIQSS123}) then all these algorithms have the same error bounds and convergence properties of QSS methods.

\section{Examples and Results} \label{sec:results}
In this section, we present two different systems on which we perform several simulations to compare the performance of the methods proposed in this work (eLIQSS and \CQSS) with that of explicit and implicit QSS methods (QSS and LIQSS), as well as with classical time-discretization methods.

In all the simulations we used the \emph{Stand Alone QSS Solver} tool~\citep{fernandez2014stand} running on an Intel Core i5-10400 CPU @ 2.90GHz Intel i5 desktop computer under Ubuntu 22.04 OS. The models used (\texttt{adr.mo} and \verb|snn_rmse.mo|) are part of the distribution so the results below can be straightforwardly reproduced.

\subsection{Advection Diffusion Reaction Model}
The first case study is based on a  one-dimensional advection–diffusion–reaction (ADR) model, representative of continuous systems exhibiting stiff behavior. Such models are commonly encountered in applications involving heat and mass transfer, as well as in problems of continuum mechanics, where the variable of interest may represent, for example, the concentration of a chemical species subject to transport and transformation.

A simple spatial discretization of the 1D ADR model results in the following set of ODEs:
\begin{equation}\label{eq:adr}
\dot x_i=-A \cdot \frac{x_i-x_{i-1}}{\Delta x}  +D \cdot \frac{x_{i+1}-2 x_i+x_{i-1}}{\Delta x^2}
+ R\cdot (x_i^2-x_i^3)
\end{equation}
for $i=2,3,...,N-1$. Here, $A$, $D$, and $R$ are the advection, diffusion, and reaction parameters, respectively. $N$ is the number of spatial sections and $\Delta x$ is the width of each section. The first ($i=1$) and last ($i=N$) points of the grid follow the following equations:
\begin{equation}\label{eq:adr_1}
\dot x_1 =-A \cdot \frac{x_1-1}{\Delta x} 
   + D \cdot \frac{x_{2}-2 x_1+1}{\Delta x^2} +
    R\cdot (x_1^2-x_1^3)
    \end{equation}
\begin{equation}\label{eq:adr_N}
\dot x_N =-A \cdot \frac{x_N-x_{N-1}}{\Delta x} 
+D \frac{2 x_{N-1}-2 x_N}{\Delta x^2}
   + R\cdot (x_N^2-x_N^3)
    \end{equation}
This system was simulated from $t_0=0$ sec. to $t_f=3$ sec. considering  a grid with $N=100$ points, a length $L=10$, and a spatial step of $\Delta x=10/N$ and the following set of parameters:  $A=1$, $D=0.1$, and $R=100$. Figure~\ref{fig:adr} illustrates a subset of the state trajectories of the system.

\begin{figure}[ht!]
    \centering
    \includegraphics[width=0.5\linewidth]{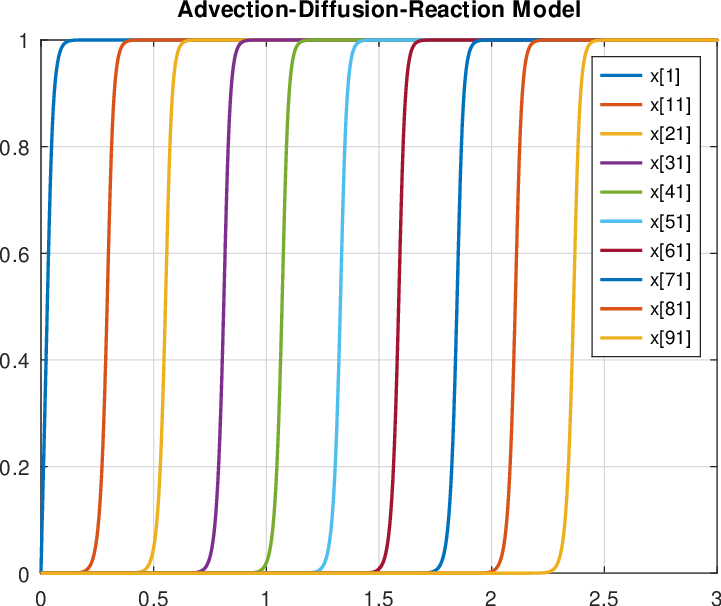}
    \caption{ADR trajectories }
    \label{fig:adr}
\end{figure}

We performed different experiments in order to compare the performance of the novel algorithms (eLIQSS and \CQSS) with that corresponding to the original definition of LIQSS and also with the classic CVODE-BDF solver. Among the classic discrete time methods implemented in the Stand Alone QSS solver (DOPRI, CVODE, IDA and DASSL), the solver that offered the best performance was CVODE-BDF due to the stiff nature of the problem and the use of a sparse representation for the Jacobian matrix.

In the different experiments we varied the tolerance settings using pairs of values $(\dQrel,\dQabs)=  \{(10^{-2},10^{-4}), (10^{-3},10^{-5}), (10^{-4},10^{-6})\}$. We also modified the  order of the QSS methods analyzing the behavior of first, second and third order algorithms. 

Tables~\ref{tab:adr_results1} and \ref{tab:adr_results2} reports the total number of steps performed by each method and the average CPU time required to complete the simulation (computed as the mean of 10 independent runs). In addition, two columns are included to show the error of each simulation measured against a high-accuracy reference solution obtained by simulating the model with CheQSS2 using very tight tolerance settings $(\dQrel,\dQabs)=(10^{-10},10^{-12})$. The errors reported are computed as the average value of the Mean Absolute Error (MAE) of each state variable. 

    \begin{table}[!h]
    \caption{Number of steps and CPU time of different numerical integration algorithms in the simulation of ADR model.}\label{tab:adr_results1}
\begin{tabular*}{\textwidth}{@{\extracolsep{\fill}}cccccc|ccc}
\toprule
 & Rel. &Abs. &\multicolumn{3}{@{}c@{}|}{N° Steps}& \multicolumn{3}{@{}c@{}}{CPU Time (ms)	}\\\cmidrule{4-6} \cmidrule{7-9}
$n$& Tol.	& Tol.&\scriptsize	\textbf{CheQSS$_n$} &\scriptsize	\textbf{eLIQSS$_n$}	& \scriptsize \textbf{LIQSS$_n$} & \scriptsize	\textbf{CheQSS$_n$} &\scriptsize	\textbf{eLIQSS$_n$}&\scriptsize \textbf{LIQSS$_n$}\\
       \midrule
 \multirow{3}{*}{1}& $10^{-2}$	&	$10^{-4}$	&	28701	&	28701	&	56464	&	4.4	&	4.6	&	8.8	\\
&$10^{-3}$&$10^{-5}$&	280812	&	280812	&	559419	&	41.6	&	41.2	&	85.3	\\
&$10^{-4}$&$10^{-6}$&		2801858	&	2801858	&	5589295	&	408.4	&	405.0	&	818.5	\\
\midrule
 \multirow{3}{*}{2}&$10^{-2}$	&$10^{-4}$&	3173	&	3644	&	4324	&	1.3	&	1.4	&	1.6	\\
&$10^{-3}$&$10^{-5}$&	8211	&	9892	&	13009	&	3.0	&	3.4	&	4.6	\\
&$10^{-4}$	&	$10^{-6}$	&23510	&	28617	&	41124	&	7.9	&	9.5	&	13.9	\\
\midrule
 \multirow{3}{*}{3}&$10^{-2}$	&$10^{-4}$	&		3345	&	2548	&	5956	&	3.1	&	2.3	&	5.2	\\
&$10^{-3}$	&$10^{-5}$&	5995	&	4012	&	9183	&	5.6	&	3.5	&	7.9	\\
&$10^{-4}$	&$10^{-6}$&12142	&	7131	&	16050	&	10.6	&	6.1	&	13.5	\\
\midrule
\midrule
&&&	\multicolumn{6}{@{}c@{}}{\small \textbf{CVODE}}  \\
 \midrule
 \multirow{5}{*}{5}&$10^{-2}$	&$10^{-4}$	&	\multicolumn{3}{@{}c@{}|}{314} 	&	\multicolumn{3}{@{}c@{}}{8.1}	\\
&$10^{-3}$	&$10^{-5}$	&	\multicolumn{3}{@{}c@{}|}{414} 	&	\multicolumn{3}{@{}c@{}}{8.2}	\\
&$10^{-4}$	&$10^{-6}$	&	\multicolumn{3}{@{}c@{}|}{649}		&	\multicolumn{3}{@{}c@{}}{10}		\\
&$10^{-5}$	&$10^{-7}$	&\multicolumn{3}{@{}c@{}|}{	870}		&	\multicolumn{3}{@{}c@{}}{15.5}\\

\botrule
    \end{tabular*}
\end{table}

   \begin{table}[!h]
   \caption{Errors of different numerical integration algorithms in the simulation of ADR model.}\label{tab:adr_results2}
\begin{tabular*}{\textwidth}{@{\extracolsep\fill}cccccc}
\toprule
   & Rel. &Abs. & \multicolumn{3}{@{}c@{}}{MAE}\\
   \cmidrule{4-6}
$n$& Tol.	& Tol.&\scriptsize \textbf{CheQSS$_n$ }&	\scriptsize \textbf{eLIQSS$_n$} &\scriptsize \textbf{LIQSS$_n$}\\
        \midrule
 \multirow{3}{*}{1}& $10^{-2}$	&	$10^{-4}$	&	1.8e-04	&	1.8e-04	&	2.2e-03	\\
&$10^{-3}$&$10^{-5}$&	2.2e-05	&	2.2e-05	&	2.3e-04	\\
&$10^{-4}$&$10^{-6}$&		2.7e-06	&	2.7e-06	&	2.3e-05	\\
\hline
 \multirow{3}{*}{2}&$10^{-2}$	&$10^{-4}$&		3.4e-04	&	5.2e-04	&	5.9e-04	\\
&$10^{-3}$&$10^{-5}$&	6.8e-05	&	3.1e-05	&	5.7e-05	\\
&$10^{-4}$	&	$10^{-6}$	&	8.6e-06	&	4.4e-06	&	5.8e-06	\\
\hline
 \multirow{3}{*}{3}&$10^{-2}$	&$10^{-4}$	&			2.8e-04	&	3.7e-04	&	2.7e-04	\\
&$10^{-3}$	&$10^{-5}$&	3.4e-05	&	3.3e-05	&	3.7e-05	\\
&$10^{-4}$	&$10^{-6}$&	4.6e-06	&	2.1e-06	&	4.2e-06	\\
\midrule
\midrule
&&&	\multicolumn{3}{@{}c@{}}{\small \textbf{CVODE}}  \\
 \midrule
 \multirow{5}{*}{5}&$10^{-2}$	&$10^{-4}$	&	\multicolumn{3}{@{}c@{}}{ 1.0e-03	}\\
&$10^{-3}$	&$10^{-5}$	&\multicolumn{3}{@{}c@{}}{	2.4e-04}\\
&$10^{-4}$	&$10^{-6}$	&	\multicolumn{3}{@{}c@{}}{2.0e-05}	\\
&$10^{-5}$	&$10^{-7}$		&	\multicolumn{3}{@{}c@{}}{2.0e-06}	\\

%&$10^{-7}$	&$10^{-9}$	&	\multicolumn{3}{c||}{205225}&	\multicolumn{3}{c||}{	858.9}	&\multicolumn{3}{c||}{1.1e-04}	\\
\botrule
    \end{tabular*}%
  % }
\end{table}
 
The following observations can be drawn from these results:

\begin{itemize}
    \item Both the extended LIQSS and the Chebyshev LIQSS methods required, in all cases, fewer integration steps and less simulation time than the standard LIQSS method for the same tolerance settings. This improvement is consistent across all tested orders. 
    
    \item CVODE was only faster than first order QSS  algorithms. For orders 2 and 3, all QSS-based methods outperformed CVODE in terms of simulation time under equivalent tolerance settings.
    
    \item Regarding the errors, all QSS methods exhibited a similar behavior with the MAE taking values on the order of magnitude of the absolute tolerance $\dQabs$ and two orders of magnitude less than the relative tolerance $\dQrel$. The only exception is the original LIQSS1, where the errors are one order of magnitude larger. This can be explained by the following fact: Since $x_i(t)$ are strictly increasing, it results that $q_i(t)$ is always larger than $x_i(t)$ introducing a disturbance that is always positive. In all the other algorithms, the disturbance changes its sign as the simulation advances because $p_i(t)$ is centered around zero (\CQSS algorithms) and/or because the higher order derivatives of $x_i$ change their sign (see the solution in Fig.\ref{fig:adr}).    
    
    \item CVODE exhibited error levels comparable to those of the original LIQSS methods, that is, approximately one order of magnitude larger than those obtained with eLIQSS and CQSS under identical tolerance settings.

    \item A direct comparison between eLIQSS and \CQSS reveals the following:
    \begin{itemize}
        \item Both first-order versions behave exactly the same, since the algorithms are identical, as anticipated in the previous section.
        \item For higher orders algorithms (2 and 3), the errors obtained with both methods are of the same order of magnitude in all tested cases. At second order, \CQSS is more efficient, requiring fewer steps and shorter simulation times. At third order, however, eLIQSS becomes more efficient. 
        \item \CQSS methods were designed to perform the largest possible steps maximizing the time that $q_i$ remains close to $x_i$. However, that maximum time only holds in first order linear system. In presence of non-linearities or when in between two steps the state trajectory $x_i$ changes its derivative due to a change in another quantized state $q_j$, the steps may result shorter. For this reason, in some settings eLIQSS performs less steps than \CQSS.
        
    \end{itemize}

\end{itemize}

\subsection{Spiking Neural Network Model}
The second example corresponds to a spiking neural network (SNN) where each neuron is described through a Leaky Integrate-and-Fire (LIF) model, based on the work presented in \citep{schmidt2018, bergonzi2024quantization}.
The state of the $i$-th neuron is characterized by two variables: the membrane potential $V_i(t)$ and the synaptic input current $I^s_i(t)$. Their temporal evolution, in the absence of spiking events, is governed by the following set of differential equations:

\begin{equation} \label{eq:subthreshold_alt}
\begin{split}
\frac{d I^s_i(t)}{dt} &= -\frac{I^s_i(t)}{\tau_s}  \\
\frac{d V_i(t)}{dt} &= -\frac{V_i(t)-E_L}{\tau_m} + \frac{I^s_i(t)}{C_m}
\end{split}
\end{equation}
for $i=1,\ldots,N_e$, where $N_e$ denotes the total number of neurons in the network.
The definitions and values of parameters $\tau_s$, $E_L$, $\tau_m$, and $C_m$ are summarized in Table~\ref{tab:param_alt}.

\begin{table}[ht]
%\centering
\caption{Parameters employed in the LIF-based spiking neuron model.}\label{tab:param_alt}
\begin{tabular}{@{}lcl@{}}
\toprule
\multicolumn{3}{@{}c@{}}{Model parameters} \\
\midrule
Value && Description \\
\cmidrule{1-1}\cmidrule{3-3}
$ \tau_m = 10$ ms && membrane time constant \\
$\tau_r = 2$ ms && absolute refractory time \\
$\tau_s = 0.5$ ms && synaptic current decay constant \\
$C_m= 250 $pF && membrane capacitance \\
$V_r= -65 $mV && reset potential \\
$ \theta =-50$ mV && firing threshold \\
$E_L=-65$ mV && leak reversal potential\\
$\nu_{bg} = 8  $ spikes/s && Average external
    spike rate \\
$k_{ext}=940$ &&  External inputs per population\\
\botrule
\end{tabular}
\end{table}

This model also incorporates discontinuous dynamics. Specifically, whenever $V_i(t)$ reaches the firing threshold $\theta$, the $i$-th neuron emits a spike, its potential is reset to the value $V_r$, and the spike signal is transmitted to its postsynaptic targets. This event triggers an instantaneous update of the synaptic currents of the receiving neurons. The update rule depends on whether the neuron that emitted the spike is excitatory or inhibitory:
\begin{equation}\label{eq:synapsis_alt}
I^s_j(t^+) \gets I^s_j(t)+J^{\mathrm{exc}}_i,
\end{equation}
for excitatory neurons, and
\begin{equation}\label{eq:synapsis2_alt}
I^s_j(t^+) \gets I^s_j(t)-J^{\mathrm{inh}}_i,
\end{equation}
for inhibitory neurons. The parameters $J^{\mathrm{exc}}_i$ and $J^{\mathrm{inh}}_i$ are the synaptic efficacies, that  follow a Gaussian distributions, consistent with the setup in \citep{bergonzi2024quantization}.

For the case study, we simulated a network of $N_e=1000$ neurons, of which $800$ are of excitatory type, and the remaining $200$ are of inhibitory type. Following the random connectivity pattern described in \citep{bergonzi2024quantization}, each neuron receives spikes from $m=10$ randomly selected synaptic connections, with $80\%$ of these connections originating from excitatory neurons and the remaining $20\%$ from inhibitory neurons. 

The initial states were randomly chosen with uniform distribution $V(t=0)\sim U[-65,-64]$ mV, $I_s(t=0)\sim U[0.4,0.5] $ nA. We also considered that each neuron receives an independent input spike train that follows a Poisson process with stationary rate $\nu_{\mathrm{ext}}=k_{ext}\cdot\nu_{bg}$. 

Figures \ref{fig:V_states} and \ref{fig:Is_states} illustrate a subset of the state trajectories of the system.

\begin{figure*}[ht]
    \centering
\begin{subfigure}[b]{0.45\textwidth}
    \includegraphics[width=\linewidth]{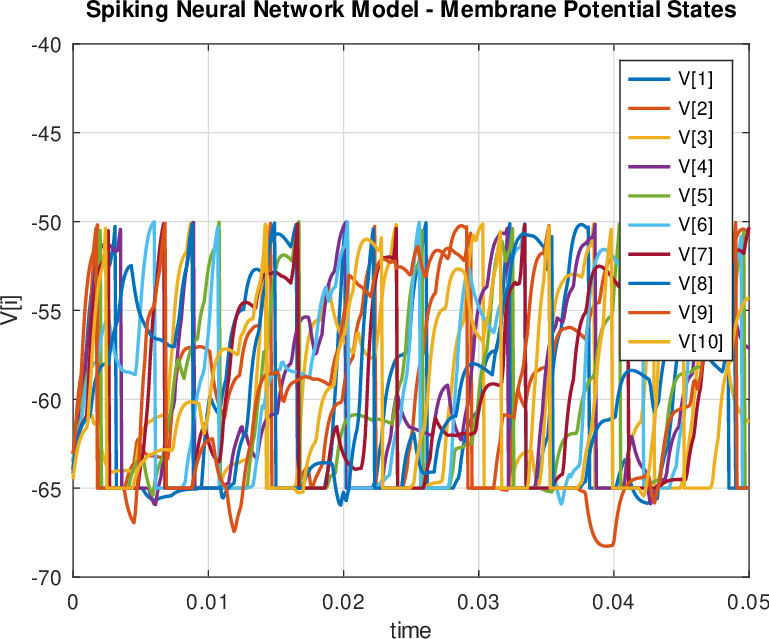}
    \caption{Membrane potentials}
    \label{fig:V_states}
\end{subfigure}    
\begin{subfigure}[b]{0.45\textwidth}
        \includegraphics[width=\linewidth]{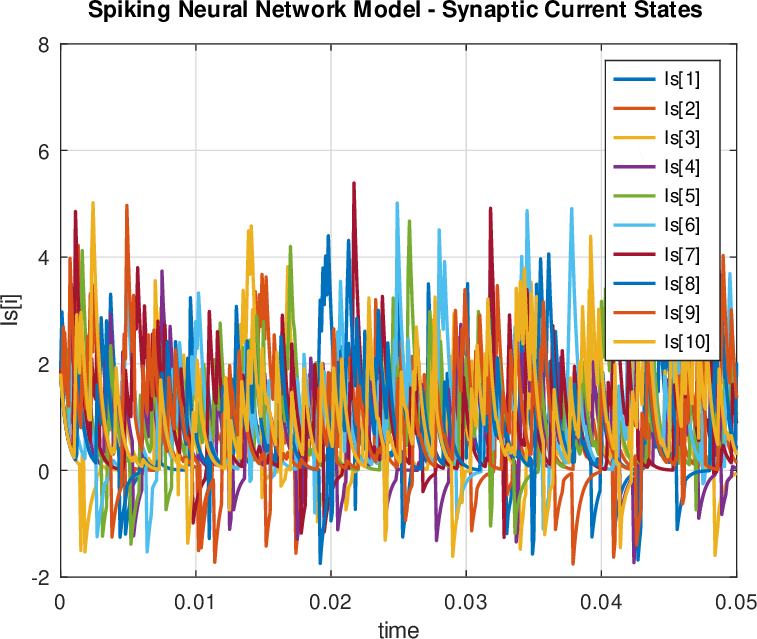}
    \caption{Synaptic currents}
    \label{fig:Is_states}
\end{subfigure}
\caption{A subset of the state trajectories for the SNN model.}
\end{figure*}

We simulated the system with $t_f=0.05$s using different QSS methods: \CQSS, eLIQSS, and QSS and the classic discrete time DOPRI solver. Since this problem is non-stiff, DOPRI exhibited the best performance among classic discrete time algorithms. Also for this reason, we used QSS instead of LIQSS among previously defined QSS algorithms. The simulation were performed for the different orders of the QSS methods (1 to 3) and using different tolerance settings. 

Since the system involves stochastic components, we selected 10 different random seeds and performed 10 simulations for each combination of numerical method, method order, and tolerance values. We then computed the average number of steps and the average simulation time across these 10 runs, as reported in Table~\ref{tab:snn_results1}. 

To evaluate the accuracy of the simulations taking into account the chaotic regime of the solutions, we use as an error metric the difference in the total number of output spikes emitted by the neurons during the simulation. This quantity provides a global indicator of accuracy, since small deviations in the trajectories can lead to significant differences in spike timing and counts. 

To obtain a reference solution, we performed an additional high-precision simulation (using CheQSS3 with \(\Delta Q = 10^{-10}\)), extracted the total number of output spikes, and used this value as a baseline. For each random seed and simulation configuration, we computed the total number of output spikes and compared it with the reference spike count. 

To quantify the simulation error, we computed the Mean Relative Error (MRE) between the spike counts obtained in each run and the reference spike count as follows:
\begin{equation}
\text{MRE} = \frac{1}{10} \sum_{i=1}^{10} \frac{\left| S^{\text{ref}}_i - S_i \right|}{S^{\text{ref}}_i},
\end{equation}
where \(S_i\) denotes the spike count obtained in the \(i\)-th simulation and \(S^{\text{ref}}_i\) the corresponding reference spike count. The resulting MRE values are reported in Table~\ref{tab:snn_results2}.

    \begin{table}[!h]
     \caption{Number of steps and CPU time of different numerical integration algorithms in the simulation of SNN model.}    \label{tab:snn_results1}
    \begin{tabular*}{\textwidth}{@{\extracolsep\fill}ccccc|ccc}
         \toprule
  &Abs. &\multicolumn{3}{@{}c@{}|}{N° Steps}  &\multicolumn{3}{@{}c@{}}{CPU Time (ms)}  \\       \cmidrule{3-8}
$n$& Tol.	&\scriptsize	\textbf{\CQSS$_n$} &\scriptsize	\textbf{eLIQSS$_n$}	& \scriptsize \textbf{QSS$_n$}&\scriptsize	\textbf{\CQSS$_n$} &\scriptsize	\textbf{eLIQSS$_n$}&\scriptsize \textbf{QSS$_n$}\\
        \midrule
 \multirow{5}{*}{2}& $10^{-1}$	&	222420	&	248389	&	487716	&	119.0	&	129.1	&	191.9	\\
&$10^{-2}$&	576976	&	676553	&	342798	&	244.6	&	279.5	&	453.0	\\
&$10^{-3}$&	1713686	&	2026289	&	4080886	&	642.1	&	748.1	&	1296.8	\\
&$10^{-4}$&	5265527	&	6284049	&	2603478	&	1877.9	&	2205.9	&	3851.6	  \\
&$10^{-5}$&	16300337	&	19749735	&	39542559	&	5592.9	&	6748.6	&	10908.7	\\
\midrule
 \multirow{5}{*}{3}&$10^{-1}$	&	156299	&	178947	&	471965	&	156.3	&	165.2	&	312.8	\\
&$10^{-2}$&234559	&	311722	&	695625	&	204.4	&	244.9	&	426.9	\\
&$10^{-3}$	&411516	&	537240	&	167814	&	317.4	&	383.7	&	657.2	\\
&$10^{-4}$&	776038	&	1063599	&	2270492	&	542.9	&	699.7	&	1211.8	\\
&$10^{-5}$&	1558635	&	2220378	&	4625227	&	1040.5	&	1401.2	&	2400.7	\\
\midrule
\midrule
&&	\multicolumn{6}{@{}c@{}}{\small \textbf{DOPRI}}  \\
 \midrule
 \multirow{5}{*}{5}&$10^{-1}$	&	\multicolumn{3}{@{}c@{}|}{59925} 	&	\multicolumn{3}{@{}c@{}}{7823.2}	\\
&$10^{-2}$	&	\multicolumn{3}{@{}c@{}|}{60185} 	&	\multicolumn{3}{@{}c@{}}{7752.9}	\\
&$10^{-3}$		&	\multicolumn{3}{@{}c@{}|}{60277}		&	\multicolumn{3}{@{}c@{}}{7627.9}		\\
&$10^{-4}$		&\multicolumn{3}{@{}c@{}|}{	60314}		&	\multicolumn{3}{@{}c@{}}{7550.2}	\\
&$10^{-5}$		&\multicolumn{3}{@{}c@{}|}{60329}	&	\multicolumn{3}{@{}c@{}}{7487.7}		\\
\botrule
    \end{tabular*}%
\end{table}

    \begin{table}[!h]
     \caption{Error of different numerical integration algorithms in the simulation of SNN model.}    \label{tab:snn_results2}
    \begin{tabular*}{0.75\textwidth}{@{\extracolsep\fill}ccccc}
         \toprule
  &Abs.  & \multicolumn{3}{@{}c@{}}{Mean Relative Error in Output Spikes} \\       
 \cmidrule{3-5}
$n$& Tol.	&\scriptsize \textbf{\CQSS$_n$ }&	\scriptsize \textbf{eLIQSS$_n$} &\scriptsize \textbf{QSS$_n$}\\
        \midrule
 \multirow{5}{*}{2}& $10^{-1}$	&	1.84E-03	&	1.40E-03	&	2.22E-03	\\
&$10^{-2}$&		4.73E-04	&	3.97E-04	&	3.03E-04	\\
&$10^{-3}$&		3.77E-05	&	3.75E-05	&	3.77E-05\\
&$10^{-4}$&		0&	0&	0  \\
&$10^{-5}$&		0	&	0	&	0	\\
\midrule
 \multirow{5}{*}{3}&$10^{-1}$	&	2.05E-03	&	1.31E-03	&	8.93E-04	\\
&$10^{-2}$&	3.80E-04	&	5.11E-04	&	3.60E-04\\
&$10^{-3}$	&	1.13E-04	&	0&	0\\
&$10^{-4}$&		0&	0&	0\\
&$10^{-5}$&		0&	0&	0\\
\midrule
\midrule
&&	\multicolumn{3}{@{}c@{}}{\small \textbf{DOPRI}}  \\
 \midrule
 \multirow{5}{*}{5}&$10^{-1}$	&		\multicolumn{3}{@{}c@{}}{3.79E-05}\\
&$10^{-2}$	&\multicolumn{3}{@{}c@{}}{	1.88E-05}\\
&$10^{-3}$	&	\multicolumn{3}{@{}c@{}}{0}	\\
&$10^{-4}$	&	\multicolumn{3}{@{}c@{}}{0}	\\
&$10^{-5}$	&	\multicolumn{3}{@{}c@{}}{0}	\\
\botrule
    \end{tabular*}%
\end{table}

Several observations can be drawn from these results:

\begin{itemize}
    \item In all cases, the new methods (\CQSS and eLIQSS) were more efficient than QSS, as evidenced by the lower number of steps and shorter total simulation times.
    
    \item Although DOPRI requires fewer steps, each step is computationally more expensive, resulting in significantly higher overall simulation times. This is because, as a classical discrete-time method, DOPRI updates all state variables at every step, whereas in state quantization–based methods, only the variables that show significant changes are updated. 
    
    \item The errors are similar across the three state quantization methods and tend to zero as the precision increases. 
    
    \item DOPRI errors tend rapidly to zero and the number of steps is almost independent on the tolerance. This is due to the fact that the step size is limited by the time between events (which correspond to neuron or input spikes). The small time between events then causes that the step size is small resulting in high accuracy solutions. However, the CPU time is about one order of magnitude larger than that of the third order QSS methods achieving the same accuracy in practice.   
    
    \item When comparing \CQSS and eLIQSS directly, \CQSS consistently requires fewer computational resources: in all tested configurations, eLIQSS uses between $10\%$ and $40\%$ more steps and between $8\%$ and $35\%$ more simulation time. Regarding the accuracy, except for the third order case using $\dQabs=10^{-3}$, both methods produce very similar errors under identical tolerance settings.
\end{itemize}

\section{Conclusions and Future Work} \label{sec:conclusions}
We presented a procedure to develop new LIQSS algorithms by designing the difference polynomial $p_i(t)=x_i(t)-q_i(t)$. Using this procedure, we first redesigned the original LIQSS methods of order 1 to 3 obtaining also an extended version (eLIQSS) that can perform larger steps. With the goal of maximizing the step size, we then designed the algorithms of Chebyshev LIQSS (\CQSS) of order 1 to 3. Taking into account that the difference polynomial verifies $|p_i(t)|=|x_i(t)-q_i(t)|\leq \Delta Q_i$, the novel algorithms (eLIQSS and \CQSS) inherit all the theoretical properties of QSS methods regarding convergence and global error bounds.

The new methods were implemented in the Stand Alone QSS Solver and were tested in two simulation examples: a space-discretized one dimensional advection-diffusion-reaction equation and a spiking neural network model. In both examples the new algorithms outperformed existing QSS methods and classic discrete time solvers.

This works opens different lines for future research. The design procedure can be applied to develop new QSS algorithms with different properties. Also, a deeper study of the performance of the new methods in different applications were QSS methods are usually convenient would be necessary to eventually replace old LIQSS algorithms by their new eLIQSS and \CQSS methods.

\section*{Declarations}
This work was partially funded by grant PICT–2021 00826 (ANPCYT).

\bibliography{references}

\end{document}